\documentclass{amsart}
\usepackage{amssymb}
\usepackage{amsfonts}
\usepackage{amstext}
\usepackage{algorithmic}
\usepackage{algorithm}
\usepackage{graphicx}
\usepackage{epstopdf}
\usepackage[all]{xy}

\allowdisplaybreaks
\numberwithin{equation}{section}
\newtheorem{theorem}{Theorem}[section]
\newtheorem{lemma}[theorem]{Lemma}
\newtheorem{proposition}[theorem]{Proposition}
\newtheorem{corollary}[theorem]{Corollary}

\theoremstyle{definition}
\newtheorem{definition}[theorem]{Definition}
\newtheorem{example}[theorem]{Example}

\theoremstyle{remark}
\newtheorem{remark}[theorem]{\bf{Remark}}

\newcommand{\R}{{\mathbb{R}}}

\newcommand{\C}{{\mathbb{C}}}
\newcommand{\Z}{{\mathbb{Z}}}
\renewcommand{\H}{{\mathbb{H}}}
\renewcommand{\O}{{\mathbb{O}}}

\newcommand{\Gcal}{{\mathcal{G}}}

\newcommand{\Rcal}{{\mathcal{R}}}
\newcommand{\Scal}{{\mathcal{S}}}

\renewcommand{\ker}{{\rm{ker}}}

\newcommand{\im}{{\rm{Im}}}

\newcommand{\tens}{\otimes}
\newcommand{\id}{{\rm id}}
\newcommand{\bo}{{}^{(1)}}
\newcommand{\bt}{{}^{(2)}}
\newcommand{\bth}{{}^{(3)}}
\renewcommand{\o}{{}_{(1)}}
\renewcommand{\t}{{}_{(2)}}
\renewcommand{\th}{{}_{(3)}}
\newcommand{\extd}{{\rm d}}
\newcommand{\del}{{\partial}}
\newcommand{\eps}{\varepsilon}
\newcommand{\rcocross}{\blacktriangleright\!\!<}

\newcommand{\iso}{\cong}

\begin{document}

\title{Hopf quasigroups and the algebraic 7-sphere}
%\thanks{}

\author{J. Klim and S. Majid}
\address{Queen Mary, University of London\\
School of Mathematics, Mile End Rd, London E1 4NS, UK}
\date{Version 5: 23rd July 2009}

\begin{abstract}
We introduce the notions of Hopf quasigroup and Hopf coquasigroup $H$ generalising the classical notion of an inverse property quasigroup $\Gcal$ expressed respectively as a quasigroup algebra $k\Gcal$ and an algebraic quasigroup $k[\Gcal]$. We prove basic results as for Hopf algebras, such as anti(co)multiplicativity of the antipode  $S:H\to H$, that $S^2=\id$ if $H$ is commutative or cocommutative, and a theory of crossed (co)products. We also introduce the notion of a Moufang Hopf (co)quasigroup and show that the coordinate algebras $k[S^{2^n-1}]$ of the parallelizable spheres are algebraic quasigroups (commutative Hopf coquasigroups in our formulation) 
and Moufang. We make use of the description of composition algebras such as the octonions via a cochain $F$ introduced in \cite{Ma99}. We construct an example $k[S^7]\rtimes\Z_2^3$ of a Hopf coquasigroup which is noncommutative and non-trivially Moufang. We use Hopf coquasigroup methods to study differential geometry on $k[S^7]$ including a short  algebraic proof that $S^7$ is parallelizable. Looking at combinations of left and right invariant vector fields on $k[S^7]$ we provide a new description of the structure constants of the Lie algebra $g_2$ in terms of the structure constants $F$ of the octonions. In the concluding section we give a new description of the $q$-deformation quantum group $\C_q[S^3]$ regarded trivially as a Moufang Hopf coquasigroup (trivially since it is in fact a Hopf algebra) but now in terms of $F$ built up via the Cayley-Dickson process. 
\end{abstract}
\maketitle 

\section{Introduction}

It is a well-known fact that the only parallelizable spheres are $S^1,S^3,S^7$. The first two are groups and it is known that $S^7$ is something weaker (a Moufang loop or Moufang quasigroup). Just as many Lie groups have an entirely algebraic description as commutative Hopf algebras, the main goal of the present paper is to develop the corresponding theory of `algebraic quasigroups' including the coordinate algebra $k[S^7]$ of the $7$-sphere. Throughout the paper we work over a field $k$ of characteristic not 2 (unless stated otherwise). The definitions are, however, obviously more general and in particular our algebraic description also provides $\Z[S^7]$ as a Moufang algebraic quasigroup, for example. A further by-product of this algebraic formulation is that it  does not in fact require the `coordinate algebra' to be commutative i.e. provides the framework to quantise the notion of inverse property quasigroups and Moufang loops in the same way as Hopf algebras provided the framework for quantum group versions of the standard Lie groups. This is not actually our present goal but we will give an example which is noncommutative and not a Hopf algebra, i.e. genuinely both noncommutative and `quasi', and will touch upon $q$-deformed examples in the concluding remarks at the end of the paper.

An outline of the paper is as follows. Section 2 provides some preliminary background needed for all the examples in the paper, namely a way of working with the octonions as twisted group quasialgebras $k_FG$ defined by the group $G=\Z_2^n$ of $\Z_2$-valued vectors and a cochain $F$. The approach is due to the 2nd author and H. Albuquerque \cite{Ma99} to which we refer for further details. The modest new result in this section is that the requirement of a composition algebra completely determines the associator  $\phi$ and symmetry $\Rcal$ in terms of linear independence over $\Z_2$, irrespective of the form of $F$. In Section 3 we provide some preliminary background on classical inverse property quasigroups. We mean by this only a set with a product, not necessarily associative, an identity and two-sided inverses in a certain strong sense. Such objects are also called `loops with inverse property'. As we only consider the case with inverses we will simply refer to quasigroups for brevity and leave the adjective `inverse property' understood throughout the paper.  We show (Proposition~3.6) how the associated unit spheres $S^{2^n-1}$ can be seen to be such quasigroups in terms of the properties of $F$. We also recall the special case of Moufang loops and provide some elementary proofs and observations mainly as a warm-up to the Hopf case to follow. A modest new feature is to study quasigroups in terms of a multiplicative associator $\varphi$ inspired by the theory of quasialgebras.

Section 4 now proceeds to `linearise' the theory of quasigroups, which now appear as cocommutative examples $k\Gcal$ (the `quasigroup algebra') of our new notion of `Hopf quasigroups'. We show that much of the elementary quasigroup theory extends to this linear setting of a coalgebra $H$ equipped with a certain but not-necessarily associative linear product. The inversion operation appears as a linear map $S:H\to H$ and we prove some basic properties analogous to theorems\cite{Ma:book} for Hopf algebras. Probably the most important is that $S$ reverses the order of products and coproducts, see Proposition~4.2, and squares to the identity if $H$ is commutative or cocommutative, see Proposition~4.3.  The notion also  includes the example of an enveloping algebra $U(L)$ of a Mal'tsev algebra recently introduced in \cite{PS}, provided one supplements the `diagonal coproduct' there with an antipode and a counit defined by $Sx=-x$ and $\eps x=0$ for all $x\in L$ (Propositions~4.8 and~4.9).  Thus our axioms  unify quasigroups and Mal'tsev algebras just as Hopf algebras historically unified groups and Lie algebras.

We are then ready, in Section 5, to `dualise' these notions to a theory of `Hopf coquasigroups' adequate to contain commutative examples  $k[\Gcal]$ (`quasigroups coordinate algebras') such as $k[S^{2^n-1}]$ (Proposition~5.7). Here the algebra $A$ is associative and equipped with a `coproduct' $\Delta: A\to A\tens A$, a counit $\eps:A\to k$ and an antipode $S:A\to A$ now playing the role of quasigroup product, identity and inversion.  The nonassociativity of the product becomes now non-coassociativity of $\Delta$. Although significantly more complicated than usual Hopf algebra theory we show again that a general theory is possible and prove some basic results as for Hopf algebras. The theory is not limited to commutative algebras and Example~5.11 provides a noncommutative Moufang Hopf coquasigroup as a genuine example of the theory.

Section 6 continues to use Hopf algebra-type methods to develop the notion of covariant differential calculus on  Hopf coquasigroups similar to that for Hopf algebras \cite{Wo89}. We obtain a short algebraic proof of the paralellizability of $S^7$ and see how the Lie algebra $g_2$ appears in terms of corresponding left and right invariant vector fields. The theory constructs this Lie algebra in terms of the structure constants $F$ of the octonions. We also see exactly how the left-invariant vector fields alone fail to close due to the nontrivial associator $\phi$.

We conclude in Section 7 with some preliminary remarks about complex generators and `quantum' Moufang loops in the sense of $q$-deformed examples where the Hopf coquasigroup becomes noncommutative as a parameter $q$ differs from 1.  We have an obvious example $\C_q[S^3]$ as the usual $*$-quantum group version of $SU_2$ but we provide a new `Cayley-Dickson' type presentation of this. These remarks suggest  a possible $\C_q[S^7]$ or other quantization to be developed in a sequel.

\section{Preliminaries on $k_FG$ approach to composition algebras}

In \cite{Ma99} the authors constructed the division algebras and higher Cayley algebras 
as twisted group rings, as follows. Let $k$ be a field, $G$ a finite group and $F:G\times G\to k^*$ a 2-cochain. This means that $F(e,a)=F(a,e)=1$ for all $a\in G$, where $e$ is the group identity.  Let \[ \phi(a,b,c)={F(a,b)F(a+b,c)\over F(b,c)F(a,b+c)}\]
 be the 3-cocycle coboundary of $F$. Finally, define $k_FG$ to be a vector space with basis $\{e_a\ :\ a\in G\}$ with product
\[ e_ae_b=F(a,b)e_{a+b}\]
Since our underlying group $G$ is going to be Abelian we will write it additively as here. Then it is easy to see that $(e_ae_b)e_c=\phi(a,b,c)e_a(e_be_c)$, i.e. $k_FG$, while not associative, is {\em quasi-associative} in the sense that its nonassociativity is strictly controlled by a 3-cocycle. In categorical terms it lives in the symmetric monoidal category of $G$-graded spaces with associator defined by $\phi$ and symmetry defined by $\Rcal(a,b)={F(a,b)\over F(b,a)}$. The choice $G=\Z_2^3$ and  a certain $F$ gives the octonions in this form.

We do not need the exact form of $F$ but rather a theorem \cite{Ma99} that $k_F\Z_2^n$ is a composition algebra with respect to the Euclidean norm in basis $G$ iff
\begin{equation} F(a,b)^2=1,\ \forall a,b \label{square}\end{equation}
\begin{equation} F(a,a+c) F(b,b+c)+F(a,b+c)F(b,a+c)=0,\ \forall a\ne b,\ \forall c \label{quad} \end{equation}
As a consequnce, the cochain $F$ will also satisfy
\begin{equation} F(a,a+b)=-F(a,b)\, \forall a\ne 0,\ \forall b \label{canr}\end{equation}
\begin{equation} F(a+b,a)=-F(b,a)\, \forall a\ne 0,\ \forall b \label{canl}\end{equation}
\begin{equation} F(a,b)F(a,c)=-F(a+b,c)F(a+c,b)\, \forall b\ne c, \ \forall a \label{quad2}\end{equation}
\begin{equation} F(a,c)F(b,c)=-F(a,b+c)F(b,a+c)\, \forall a\ne b, \ \forall c \label{quad3}\end{equation}
These identities are obtained from (\ref{quad}), for example, setting $b=0$ gives (\ref{canr}). We also know and will use that 
\begin{equation}F(a,a)=-1,\quad\forall a\ne 0\label{Fxx}\end{equation}
 in our examples. This applies to all the division algebras  $\C,\H,\O$ given by $k=\R$ and such $F$ on $\Z_2^n$ for $n=1,2,3$  respectively. Note that the Octonions being division algebras have left and right cancellation and obey the three equivalent  Moufang loop identities which we will recall later.

\begin{lemma} The choice $G=\Z_2^n$ and any $F$ giving a composition algebra, $k_FG$, with
\[ \phi(a,b,c)=\begin{cases}-1 & a,b,c{\rm \ linearly\ independent\ as\ vectors\ over\ }\Z_2\\ 1 & {\rm otherwise}\end{cases}\]
\[ \Rcal(a,b)=\begin{cases}-1 & a,b,a+b\ne 0 \\ 1 & {\rm otherwise}\end{cases}=\begin{cases}-1 & a,b{\rm \ linearly\ independent\ as\ vectors\ over\ }\Z_2 \\ 1 & {\rm otherwise}\end{cases}\]
In particular, $\phi$ and $\Rcal$ are symmetric, and $\phi(a+b,b,c)=\phi(a,b,c)$ and $\Rcal(a+b,b)=\Rcal(a,b)$ for all $a,b\in G$.
\label{phi}
\end{lemma}

\proof We start with the symmetry $\Rcal$. If $a=b=0$,

\[ \Rcal(a,b)=\Rcal(0,0)=F(0,0)F(0,0)=1, \]

and if $a=b$ then

\[ \Rcal(a,b)=\Rcal(a,a)=F(a,a)F(a,a)=1. \]

For the case $a\ne b$, one can use the composition identity, (\ref{quad}), by setting $x=a,\ y=b, \ z=0$ to obtain

\[ F(a,a)F(b,b)=-F(a,b)F(b,a)=-\Rcal(a,b)\]

If $a=0, b\ne 0$ then $F(a,a)=1,\ F(b,b)=-1$, hence $\Rcal(a,b)=1$, similarly, if $b=0,a\ne 0$, then $\Rcal(a,b)=1$. Finally, if $a,b\ne 0$, $F(a,a)=F(b,b)=-1$, hence $\Rcal(a,b)=-1$. This establishes the stated form of $\Rcal$.

For $\phi$ we again consider the cases. Suppose $a,b,c$ are linearly dependent, say $a=b+c$ with $a,b,c\ne 0$, then,
\[\phi(a,b,c)=\phi(b+c,b,c)=F(b+c,b)F(c,c)F(b,c)F(b+c,b+c)=F(b+c,b)F(b,c),\]
since $F(c,c)=F(b+c,b+c)=-1$. Using identity (\ref{canl}), set $x=b\ne 0,\ y=c$ to obtain $F(b+c,b)=-F(c,b)$. Thus,
\[\phi(b+c,b,c)=F(b+c,b)F(b,c)=-F(c,b)F(b,c)=-\Rcal(b,c)=1\]
Now suppose, $a=b\ne 0$, then,
\[ \phi(a,b,c)=\phi(a,a,c)=F(a,a)F(0,c)F(a,c)F(a,a+c)=-F(a,c)F(a,a+c)\]
Using (\ref{canr}) and (\ref{square}), we get
\[\phi(a,a,c)=F(a,c)F(a,c)=1\]
Similarly if $a=c\ne 0$ or $b=c\ne 0$. Next, suppose $a=0$, then using (\ref{square})
\[\phi(a,b,c)=\phi(0,b,c)=F(0,b)F(b,c)F(b,c)F(0,b+c)=F(b,c)F(b,c)=1.\]
Finally, suppose $a,b,c\ne 0$ are linearly independent, then
\[\phi(a,b,c)=F(a,b)F(a+b,c)F(b,c)F(a,b+c).\]
Using identities (\ref{quad2}) and (\ref{quad3}), we obtain
\[ F(a,b)F(a+b,c)=-F(a,c)F(a+c,b),\]
\[ F(b,c)F(a,b+c)=-F(a,c)F(b,a+c),\]
hence,
\[\phi(a,b,c)=F(a,c)F(a+c,b)F(a,c)F(b,a+c)=\Rcal(a+c,b)=-1.\]
This establishes the stated form of $\phi$. 

Finally, we show that $\phi$ is symmetric. Assume $a,b,c\ne 0$ are linearly independent in $\Z_2^n$. Using the properties of $\Rcal$, that $G$ is abelian, and identity (\ref{quad3}),
\begin{eqnarray*}
\phi(a,b,c)	&=	&F(a,b)F(a+b,c)F(b,c)F(a,b+c)\\
						&=	&-F(b,a)F(b+a,c)F(b,c)F(a,b+c)\\
						&=	&F(b,a)F(b+a,c)F(a,c)F(b,a+c)\\
						&=	&\phi(b,a,c)
\end{eqnarray*}
\begin{eqnarray*}
\phi(a,b,c)	&=	&F(a,b)F(a+b,c)F(b,c)F(a,b+c)\\
						&=	&F(b,a)F(c,b+a)F(c,b)F(c+b,a)\\
						&=	&\phi(c,b,a)
\end{eqnarray*}

The final conclusions follow immediately as the linear independence of $a+b,b,c$ depends on the linear independence of $a,b,c$.
\endproof

The statement that $k_FG$ is a composition algebra means that the norm

\[ q(\sum_a u_ae_a)= \sum_a u_a^2\]

is multiplicative. In particular, it means that the set of elements of unit norm, i.e. the $2^n-1$-spheres over $k$, is closed under multiplication in $k_FG$. We denote this set by 

\[ \Scal^{2^n-1}=\{\sum_a u_a e_a|\quad \sum_a u_a^2 =1\}\subset k_F\Z_2^n\]

which becomes a usual sphere if we work over $\R$. The next lemma makes it clear that we have 2-sided inverses by exhibiting them.

\begin{lemma} If $u\in k_FG$ has unit norm then 
\[ u^{-1}=u_0 e_0 - \sum_{a\ne 0} u_a e_a=2u_0-u\]
\end{lemma}

\proof 
\[ (u_0e_0-\sum_{a\ne 0}u_ae_a)(u_0e_0+\sum_{b\ne 0}u_be_b)=\sum u_a^2e_0+\sum_{a,b,a+b\ne 0}u_au_bF(a,b)e_{a+b}=q(u)\]
since under the assumptions of the last summation we know that $F(a,b)=-F(b,a)$. Similarly on the other side. \endproof

We also have $S^{2^n-1}\subset k_F^\times G$, the set of elements of nonzero norm. These are also closed under the product of $k_FG$ and from the lemma we deduce that they are invertible with $u^{-1}=(2u_0-u)/q(u)$. Over  $\R$ this larger object is the set of the invertible elements of $k_FG$.  Similarly we define a finite object $\Gcal_n\subset S^{2^n-1}$ by
\[ \Gcal_n=\{\pm e_a\ |\ a\in \Z_2^n\}\subset k_F\Z_2^n\]
The elements here all have unit norm and since $F(a,b)=\pm 1$ we see that $\Gcal_n$ is closed under multiplication and has identity $1=e_0$. The inverses are $e_i^{-1}=-e_i$ for $i\ne 0$.  In the case $n=3$ one has $\Gcal_n=\Gcal_\O$ the order 16 Moufang loop associated to the octonions. We explain the terminology and background next.

\section{Quasigroups}

In this section we provide some elementary background on loops and quasigroups, mainly as a warm-up for the next section. The term `loop' refers to a set $\Gcal$ with a product such that the operations $L_u$ and $R_u$ of left and right multiplication by $u$ are bijective for all $u\in \Gcal$. In addition we require an identity element $e\in \Gcal$ for the product. We will be interested only in loops with two-sided inverse property or `IP loops'. In the literature the term `quasigroup' is also used to denote a set with $L_u,R_u$ bijective, so our objects could also be called `quasigroups with 2-sided inverse property', but we will just refer to quasigroups for short. 

\begin{definition} We define (an inverse property) {\em quasigroup} (or `IP loop') as a set $\Gcal$ with a product, identity $e$ and with the property that for each $u\in \Gcal$ there is $u^{-1}\in\Gcal$ such that 
\[ u^{-1}(uv)=v,\quad (vu)u^{-1}=v,\quad \forall v\in \Gcal.\]
A quasigroup is {\em flexible} if $u(vu)=(uv)u$ for all $u,v\in \Gcal$ and {\em alternative} if also $u(uv)=(uu)v$, $u(vv)=(uv)v$ for all $u,v\in \Gcal$. It is called {\em Moufang} if $u(v(uw))=((uv)u)w$ for all $u,v,w\in\Gcal$. \end{definition}

The special cases are in line with usual terminology for flexible and alternative algebras, except that in algebras over characteristic not 2, any two of the three alternative conditions imply the third.

It is easy to see (and well-known) that in any quasigroup $\Gcal$ one has unique inverses and
\begin{equation} (u^{-1})^{-1}=u, \quad (uv)^{-1}=v^{-1}u^{-1},\quad \forall u,v\in \Gcal.\end{equation}
As we will generalise this in the next section, we give the proof here. We have $(u^{-1}(uv))(uv)^{-1}=u^{-1}$, but the left hand side also
simplifies by the quasigroup identities, so that $v(uv)^{-1}=u^{-1}$. Now apply $v^{-1}$ to the left of both sides. 

Moreover, note also that a Moufang loop necessarily has two-sided inverses and is therefore a Moufang quasigroup, and that the following well-known lemma holds

\begin{lemma} Let $\Gcal$ be a quasigroup, then the following identities are equivalent, for all $u,v,w\in\Gcal$
\begin{enumerate}
\item $u(v(uw))=((uv)u)w$ 
\item $((uv)w)v=u(v(wv))$ 
\item $(uv)(wu)=(u(vw))u$ 
\end{enumerate}
\end{lemma}
\proof We provide an elementary proof that will serve as a template for the Hopf case in the next section. 
Suppose (1) holds. Taking the inverse of both sides gives
\[ ((w^{-1}u^{-1})v^{-1})u^{-1}=w^{-1}(u^{-1}(v^{-1}u^{-1})) \]
for all $u,v,w\in \Gcal$, which is equivalent to (2). Similarly (2) implies (1).

Now, suppose (1) holds, and replace $w$ by $u^{-1}w$ to get $u(vw)=((uv)u)(u^{-1}w)$, therefore we obtain
\[ (uv)u=(u(vw))(u^{-1}w)^{-1}=(u(vw))(w^{-1}u) \]
Now replace $v$ by $vw$ and $w$ by $w^{-1}$ to obtain
\[ (u(vw))u=(uv)(wu) \]
for all $u,v,w\in\Gcal$, which is identity (3).

Finally, assume (3) holds, then
\[ uv=((u(vw))u)(wu)^{-1}=((u(vw))u)(u^{-1}w^{-1}) \]
Replacing $v$ by $vw$ and $w$ by $w^{-1}$ gives $u(vw)=((uv)u)(u^{-1}w)$. Finally, replacing $w$ by $uw$ gives
\[ u(v(uw))=((uv)u)w \]
which is identity (1).
\endproof

Hence Moufang implies alternative by looking at special cases the three versions. These are usual notions and results except that we have changed the emphasis by starting with 2-sided inverses, as our focus is on group theory.   

\begin{lemma}
Let $\Gcal$ be a flexible quasigroup, then
\[ u(vu^{-1})=(uv)u^{-1} \quad \forall u,v\in \Gcal\]
\end{lemma}
\proof
By flexibility and quasigroup properties, we have
\[ (u(vu^{-1}))u=u((vu^{-1})u)=uv \]
But also,
\[ ((uv)u^{-1})u=uv\]
Therefore
\[ (u(vu^{-1}))u=((uv)u^{-1})u \]
for all $u,v\in \Gcal$, hence $u(vu^{-1})=(uv)u^{-1}$ for all $u,v\in \Gcal$.
\endproof

Also for any quasigroup  $\Gcal$ we now introduce the `multiplicative associator'  $\varphi:\Gcal^3\to \Gcal$ by 
\[ (uv)w=\varphi(u,v,w)(u(vw)),\quad \forall u,v,w\in\Gcal\]
and in view of the above, we can also obtain it explicitly from
\[ \varphi(u,v,w)=((uv)w)(u(vw))^{-1}=((uv)w)((w^{-1}v^{-1})u^{-1}),\quad \forall u,v,w\in\Gcal.\]
We define the {\em group of associative elements or `nucleus'} $N(\Gcal)$ by
\[ N(\Gcal)=\{a\in\Gcal\ |\ (au)v=a(uv),\ u(av)=(ua)v,\ (uv)a=u(va),\quad \forall u,v\in\Gcal\}.\]
It is easy to see that this is indeed closed under the product and inverse operations and hence a group. We say that a quasigroup is {\em quasiassociative} if $\varphi$ and all its conjugates $u\varphi u^{-1}$ have their image in $N(\Gcal)$, and {\em central} if the image of $\varphi$ is in the centre $Z(\Gcal)$.

\begin{proposition}
Let $\Gcal$ be a  quasigroup with identity $e$ and multiplicative associator $\varphi$. Then
\begin{enumerate}
\item $N(\Gcal)=\{a\in \Gcal\ |\ \varphi(a,u,v)=\varphi(u,a,v)=\varphi(u,v,a)=e,\  \forall u,v\in \Gcal\}$.
\item  $\varphi(e,u,v)=\varphi(u,e,v)=\varphi(u,v,e)=\varphi(u,u^{-1},v)=\varphi(u,v,v^{-1})=\varphi(uv,v^{-1},u^{-1})=\varphi(u^{-1},uv,v^{-1})=\varphi(v^{-1},u^{-1},uv)=e$ $\forall u,v\in\Gcal$.
\item $\varphi(au,v,w)=a\varphi(u,v,w)a^{-1}$, $\varphi(ua,v,w)=\varphi(u,av,w)$, $\varphi(u,va,w)=\varphi(u,v,aw)$, $\varphi(u,v,wa)=\varphi(u,v,w)$ $\forall u,v,w\in\Gcal$ and $a\in N(\Gcal)$. 
\item  If  $\Gcal$ is quasiassociative then $\varphi$ is an `adjoint 3-cocycle' in the sense 
\[ \varphi(u,v,w)\varphi(u,vw,z)(u\varphi(v,w,z)u^{-1})=\varphi(uv,w,z)\varphi(u,v,wz),\quad\forall u,v,w,z\in\Gcal.\]
\item  $\Gcal$ is flexible {\em iff} $\varphi(u,v,u)=e$ for all $u,v\in\Gcal$ and alternative if also $\varphi(u,u,v)=\varphi(u,v,v)=e$ for all $u,v\in \Gcal$.
%\item $\Gcal$ is Moufang {\em iff} $u\varphi(v,u,w)u^{-1}=\varphi(u,vu,w) \quad \forall u,v,w\in\Gcal$.
\end{enumerate}
\end{proposition}
\proof
Parts (1)-(2) are immediate from the definition of $\varphi$, cancellation in the quasigroup and the inverse of $uv$. Part (3) follows from expanding $(au)(vw)$ in two ways, on the one hand as $a(u(vw))=a(\varphi(u,v,w)((uv)w))=(a\varphi(u,v,w))((uv)w)$, and on the other hand as 
\[ \varphi(au,v,w)(((au)v)w)=\varphi(au,v,w)((a(uv)w)=\varphi(au,v,w)(a((uv)w))=(\varphi(au,v,w)a)((uv)w).\] Comparing and cancelling the factor $((uv)w)$ gives the result. The other forms of the equation are similar. Part (4)  follows similarly by comparing
\begin{eqnarray*}
((uv)w)z	&	=	& (\varphi(u,v,w)(u(vw)))w = \varphi(u,v,w)((u(vw))z) = \varphi(u,v,w)(\varphi(u,vw,z)(u((vw)z)))\\
	&	=	& \varphi(u,v,w)\varphi(u,vw,z)(u((vw)z))	=	\varphi(u,v,w)\varphi(u,vw,z)(u(\varphi(v,w,z)(v(wz))))\\
	&	=	& \varphi(u,v,w)\varphi(u,vw,z)((u\varphi(v,w,z))(v(wz))) = \varphi(u,v,w)\varphi(u,vw,z)((\varphi_u(v,w,z)u)(v(wz)))\\
	&	=	& \varphi(u,v,w)\varphi(u,vw,z)(\varphi_u(u(v(wz)))) = \varphi(u,v,w)\varphi(u,vw,z)\varphi(v,w,z)(u(v(wz)))\\
((uv)w)z	&	=	& \varphi(uv,w,z)((uv)(wz)) = \varphi(uv,w,z)(\varphi(u,v,wz)(u(v(wz))))\\
	&	=	& \varphi(uv,w,z)\varphi(u,v,wz)(u(v(wz)))
\end{eqnarray*}
where $\varphi_u(v,w,z)=u\varphi(v,w,z)u^{-1}$. We use that $\varphi\in N(\Gcal)$ but we also need for the 4th equality that $\varphi_u\in N(\Gcal)$, which we have included in the definition of quasiassociativity. Part (5) is again immediate.
\endproof

Part (3) makes it clear that $N(\Gcal)$ is indeed a group (eg set $u=a^{-1}$ in the first equality). Clearly many further results are possible along these lines, for example:

\begin{proposition} A quasiassociative quasigroup $\Gcal$ is Moufang {\em iff} it is alternative and one of
\[  u\varphi(v,u,w)u^{-1}=\varphi(u,vu,w)^{-1},\quad \varphi(u,vw,v)=\varphi(u,v,w)^{-1},\quad \varphi(uv,w,u)=\varphi(u,v,w) \]
holds for all $u,v,w\in \Gcal$. 
\end{proposition}
\proof We use parts (4),(5) of the above and the three equivalent versions of the Moufang identities. \endproof

Clearly the invertible elements of the octonions $\O$ form a quasigroup, as this follows from the fact that they are a Moufang loop. The $S^7$ of unit octonions and $\Gcal_\O$ are therefore subquasigroups. In our theory these facts are easily proven from properties of $F$. The notion of sub-quasigroup here is the obvious one and note that $\varphi$ on the subquasigroup is the restriction of $\varphi$ on the larger one. 

\begin{proposition} For any $F$ on $\Z_2^n$ with the composition property, $S^{2^n-1}$ is a quasigroup and $\Gcal_n$ is a sub-quasigroup. Moreover, $\Gcal_n$ is central and quasiassociative and $\varphi(\pm e_a,\pm e_b,\pm e_c)=\phi(a,b,c)$, reproducing the coboundary $\phi=\del F$. 
\end{proposition}
\proof Working in $k_FG$ we have inverses for $S^{2^n-1}$ given by $(u^{-1})_a=F(a,a)u_a$ as explained at Lemma~2.2. One can directly verify the slightly stronger quasigroup identity by direct computation inside $k_FG$. Thus
\[ u^{-1}(uv)=\sum_{a,b,c} u_au_bv_c e_{a+b+c} F(a,a)\phi(a,b,c)F(a,b)F(a+b,c) \]
Consider the case when $a=b$, this gives
\[ \sum_{a,c}u_a^2v_c e_c F(a,a)\phi(a,a,c)F(a,a)F(0,c)=\sum_{a,c}u_a^2v_c e_c =\sum_c v_c e_c=v \]
Now consider the case when $a\ne b$. We claim that the term with given values for $a$ and $b$ cancels with the term with $a'=b,b'=a$. These give, respectively,
\[ \sum_c u_au_bv_c e_{a+b+c} F(a,a)\phi(a,b,c)F(a,b)F(a+b,c)\]
\[ \sum_c u_bu_av_c e_{b+a+c} F(b,b)\phi(b,a,c)F(b,a)F(b+a,c)\]
When $a=0$ and hence $b\ne 0$, these become
\[ \sum_c u_0u_bv_c e_{b+c} F(b,c)\]
\[ \sum_c u_bu_0v_c e_{b+c} F(b,b)F(b,c)= - \sum_c u_0u_bv_c e_{b+c} F(b,c)\]
which cancel. When $a,b\ne 0$, thse become
\[ -\sum_c u_au_bv_c e_{a+b+c} \phi(a,b,c)F(a,b)F(a+b,c)\]
\[ - \sum_c u_bu_av_c e_{b+a+c} \phi(b,a,c)F(b,a)F(b+a,c)=\sum_cu_au_bv_c e_{a+b+c} \phi(a,b,c)F(a,b)F(a+b,c)\]
which also clearly cancel. Hence,
\[\sum_{a,b,c} u_au_bv_c e_{a+b+c} F(a,a)\phi(a,b,c)F(a,b)F(a+b,c) = v = \eps(u)v\]
as required. We will later need an explicit formula for $\varphi$, easily computed from the second version of $\varphi$ as 
\begin{eqnarray*}
\varphi(u,v,w)	&	=	&	\sum_{a,b,c,a',b',c'} u_av_bw_c u_{a'} v_{b'} w_{c'} F(a',a')F(b',b')F(c',c')F(b,c)F(a,b+c)\\
	&		& F(b',a')F(c',b'+a')F(a+b+c,a'+b'+c') e_{a+b+c+a'+b'+c'}.
\end{eqnarray*}
To prove the restriction of $\varphi$ to $\Gcal_n$ we compute from this, or obtain it directly as
\begin{eqnarray*}
\varphi(e_a,e_b,e_c)	&=	&-(e_a(e_be_c))(e_c(e_be_a))\\
									&=	&-F(b,c)F(a,b+c)F(b,a)F(c,b+a)F(a+b+c,a+b+c) \\
									&=	&-\phi(a,b,c)\Rcal(a,b)\Rcal(a+b,c)F(a+b+c,a+b+c)
\end{eqnarray*}
Now, we need only consider the case when $a,b,c\ne 0$ as the trivial cases clearly coincide. For the same reason we can assume that $a,b,c$ are distinct. In that case $\Rcal(a,b)=-1$ cancels the - sign at front. Also, $a+b\ne 0$ and if $a+b+c\ne 0$ we have $\Rcal(a+b,c)=-1$ and $F(a+b+c,a+b+c)=-1$. If $a+b=c$ we have both of these factors $+1$. Hence in call cases the right hand side is $\phi(a,b,c)$. \endproof

The same result as for $S^{2^n-1}$ also applies to $k^\times_FG$ in the composition case (we just keep track of $q(a)$). Meanwhile, the result for $\Gcal_n$ can also be obtained `constructively' as a special case $C=\{\pm 1\}\subset k^*$ of the following general construction (the conditions on $F$ following from (\ref{square})-(\ref{Fxx})). We write all groups multiplicatively here.

\begin{proposition} Let $F$ be any unital 2-cochain on a group $G$ with values in an Abelian group $C$ and $\Gcal_F=C\times G$ with product $(\lambda,a)(\mu,b)=(\lambda\mu F(a,b),ab)$ for all $a,b\in G$ and $\lambda,\mu\in C$. Then 

\begin{enumerate}\item $\Gcal_F$  is a quasigroup {\em iff} $\phi(a^{-1},a,b)=\phi(b,a^{-1},a)=1$ for all $a,b\in G$, where $\phi=\del F$.
\item This happens {\em iff} $F(a,b)F(a^{-1},ab)=F(ba^{-1},a)F(b,a^{-1})=F(a^{-1},a)$ for all $a,b\in G$. 
\item The quasigroup $\Gcal_F$ is central, quasiassociative, and $C\hookrightarrow \Gcal_F\twoheadrightarrow G$ as (quasi)groups. 
\end{enumerate}
\end{proposition}
\proof $F$ a untal 2-cochain just means $F(a,e)=F(e,a)=1$ (the identity of $C$) for all $a\in G$. The identity in $\Gcal$ is $(1,e)$ and the inverse if it exists (from the left, say) must then be $(\lambda,a)^{-1}=({\lambda^{-1}\over F(a^{-1},a)},a^{-1})$. The two quasigroup identities reduce to those for $F$ stated in (2) and also imply that this is a right inverse. We interpret (2) as conditions (1) on $\phi$ from its definition as $\phi=\del F$. For the last part, clearly $C\subset \Gcal_F$ is central as $F(a,e)=F(e,a)=1$ and quasiassociative because   any element of $C$ associates with all elements of $\Gcal_F$ as $\phi=1$ if any argument is $e$, which again follows from $F$ unital. \endproof

\section{Hopf quasigroups}

In this section we `linearize' the notion of (an inverse property) quasigroup in the same way that a quantum group or Hopf algebra linearises the notion of a group, i.e we develop axioms for `quantum quasigroups' or `quantum IP loops' if one prefers. While we believe that `truly quantum' examples not obtained from quasigroups exist, our present goal is only to formalise some axioms as a warm-up to the next section. The theory at this level works over any field $k$, or with care over a commutative ring.

\begin{definition}  A \textit{Hopf quasigroup} is a possibly-nonassociative but unital algebra $H$ equipped with algebra homomorphisms $\Delta:H\to H\tens H$, $\eps:H\to k$ forming a coassociative coalgebra and a map $S:H\to H$ such that
\begin{equation*} m(\id\tens m)(S\tens\id\tens\id)(\Delta\tens\id)=\eps\tens \id = m(\id\tens m)(\id\tens S\tens\id)(\Delta\tens\id) \end{equation*}
\begin{equation*} m(m\tens\id)(\id\tens S\tens\id)(\id\tens\Delta)=\id\tens\eps = m(m\tens\id)(\id\tens\id\tens S)(\id\tens\Delta) \end{equation*}
A Hopf quasigroup is \textit{flexible} if 
\[ h\o(gh\t)=(h\o g)h\t \quad \forall h,g\in H \]
and \textit{alternative} if also 
\[ h\o(h\t g)=(h\o h\t)g,\quad h(g\o g\t)=(hg\o)g\t \quad \forall h,g\in H\]
$H$ is called \textit{Moufang} if 
\[h\o(g(h\t f))=((h\o g)h\t)f \quad \forall h,g,f\in H\]
\end{definition}

Coassociative coalgebra here means in the usual sense
\[(\id\tens\Delta)\Delta=(\Delta\tens\id)\Delta,\quad (\id\tens\eps)\Delta=(\eps\tens\id)\Delta=\id\]
as for Hopf algebras\cite{Ma:book}. We use the usual `Sweedler' notation for coalgebras $\Delta h=h\o\tens h\t$ etc.

\begin{proposition} In any Hopf quasigroup one has
\begin{enumerate}
\item $m(S\tens\id)\Delta=1.\eps=m(\id\tens S)\Delta$
\item $S$ is antimultiplicatve $S(hg)=(Sg)(Sh)$ for all $h,g\in H$
\item $S$ is anticomultiplicative $\Delta(S(h))=S(h\t)\tens S(h\o)$ for all $h\in H$. 
\end{enumerate}
Hence a Hopf quasigroup is a Hopf algebra {\em iff} its product is associative.
\label{antipode}
\end{proposition}
\proof
(1) is obtained by applying the first identity in the definition of a Hopf quasigroup to $(h\tens 1)$.
To prove (2), we consider $S(g\o)((S(h\o)(h\t g\t))S(h\th g\th))$. Using identites in the definition of a Hopf quasigroup, on the one hand this equals
\[ S(g\o)((S(h\o)(hg)\t\o)S((hg)\t\t)) = S(g\o)S(h\o)\eps(hg) = S(g)S(h)\]
While on the other hand, this equals
\[ S(g\o\o)((S(h\o\o)(h\o\t g\o\t))S(h\t g\t)) = S(g\o\o)(g\o\t S(hg\t)) = S(hg)\]
This holds for all $h,g\in H$, as required.

To prove (3) we consider $S(h\t)(h\th S(h_{(5)})\o)\tens S(h\o)(h_{(4)}S(h_{(5)})\t)$. On one hand, this equals
\[ S(h\o\t)(h\t\o\o S(h\t\t)\o)\tens S(h\o\o)(h\t\o\t S(h\t\t)\t) \]
Using that $\Delta$ is an algebra homomorphism, and (1), this equals
\[ S(h\o\t)\eps(h\t)\o \tens S(h\o\o)\eps(h\t)\t = S(h\t)\tens S(h\o) \]
Now, we can also write the original expression as
\[ S(h\o\o\t\o)(h\o\o\t\t S(h\t)\o)\tens S(h\o\o\o)(h\o\t S(h\t)\t) \]
Using properties of the Hopf quasigroup, this equals
\[ S(h\t)\o \tens S(h\o\o)(h\o\t S(h\t)\t) = S(h)\o\tens S(h)\t = \Delta(S(h))\]
This holds for all $h\in H$, hence $\Delta(S(h))=S(h\t)\tens S(h\o)$, as required.
\endproof

\begin{proposition}
Let $H$ be a Hopf quasigroup, then $S^2=\id$ if $H$ is commutative or cocommutative.
\end{proposition}
\proof
Let $h\in H$. Since $S$ is anticomultiplicative, we have
\[ S^2(h) = S^2(h\o)(S(h\t)h\th)=S(S(h\o)\t)(S(h\o)\o h\t)\]
If $H$ is commutative, that is $hg=gh$ for all $h, g\in H$, we find
\[ =(h\t S(h\o)\o)S(S(h\o)\t) = h\t\eps(S(h\o)) = h\]
using $S(h\o)(h\t g)=\eps(h)g$ for all $h,g\in H$. If $H$ is cocommutative, that is $h\o\tens h\t=h\t \tens h\o$ for all $h\in H$, we find
\[ = S(S(h\o)\o)(S(h\o)\t h\t) = \eps(S(h\o))h\t = h\]
using $(gh\o)S(h\t)=g\eps(h)$ for all $h,g\in H$. This holds for any $h\in H$, hence $S^2=\id$, as required.
\endproof

We note that if $H$ is a Moufang Hopf quasigroup with invertible antipode, then as in the classical case, the Moufang identity is equivalent to two other versions.

\begin{lemma}
Let $H$ be a Hopf quasigroup such that $S^{-1}$ exists, then the following identities are equivalent for all $h,g,f\in H$,
\begin{enumerate}
\item $h\o(g(h\t f))=((h\o g)h\t)f$
\item $((hg\o)f)g\t=h(g\o(fg\t))$
\item $(h\o g)(fh\t)=(h\o(gf))h\t$
\end{enumerate}
\label{Moufang}
\end{lemma}
\proof
Suppose (1) holds. Applying $S$ to both sides and using that $S$ is anticomultiplicative, we find
\[ ((S(f)S(h)\o)S(g))S(h)\t=S(f)(S(h)\o(S(g)S(h)\t)) \]
for all $h,g,f\in H$, which is clearly equivalent to (2). Similarly, (2) implies (1).

Assume (1) holds, then
\[ u(vw)=u\o\o(v(u\o\t(S(u\t)w)))=((u\o\o v)u\o\t)(S(u\t)w)\]
Therefore, by replacing $v$ with $vw\o$, and $w$ with $S(w\t)$, we obtain
\[ uv\eps(w) = u((vw\o)S(w\t)) = ((u\o\o(vw\o))u\o\t)S(w\t u\t)\]
Now replace $u$ by $u\o$, $w$ by $w\o$, and multiply on the right by $w\t u\t$ to obtain
\[ \eps(w\o)(u\o v)(w\t u\t)= (((u\o\o\o(vw\o\o))u\o\o\t)S(w\o\t u\o\t))(w\t u\t) \]
The LHS equals $(u\o v)(wu\t)$, so we consider the RHS. Using coassociativity and that $\Delta$ is an algebra homomorphism, this equals
\[ (((u\o\o(vw\o))u\o\t)S((w\t u\t)\o))(w\t u\t)\t =(u\o(vw))u\t \]
So we have $(u\o v)(wu\t)=(u\o(vw))u\t$, which is identity (3).

Assume (3) holds, then
\[ ((h\o\o g)(f\o h\o\t))S(f\t h\t)=((h\o\o(gf\o))h\o\t)S(f\t h\t) \]
Using that $\Delta$ is an algebra homomorphism, and defining properties of a quasigroup, the LHS equals $hg\eps(f)$, so we obtain
\[ hg\eps(f)=((h\o\o(gf\o))h\o\t)(S(h\t)S(f\t)) \]
Therefore,
\[ h(gS(f\o))\eps(f\t) = ((h\o\o((gS(f\o))f\t\o))h\o\t)(S(h\t)S(f\t\t)) \]
which simplifies to
\[ h(gS(f))= ((h\o\o g)h\o\t)(S(h\t)S(f)) \]
Replacing $f$ by $S^{-1}(h\t f)$, we obtain
\[ h\o(g(h\t f)) = ((h\o\o\o g)h\o\o\t)(S(h\o\t)(h\t f))=((h\o g)h\t)f \]
which is identity (1).
\endproof

\begin{lemma}
Let $H$ be a cocommutative flexible Hopf quasigroup, then
\[ h\o(gS(h\t))=(h\o g)S(h\t) \]
for all $h,g\in H$.
\end{lemma}
\proof
By cocommutativity and flexibility we have
\[(h\o\o(gS(h\o\t)))h\t = h\o((gS(h\t\o))h\t\t)\] 
Using the Hopf quasigroup identity on $h\t$, this equals $hg$. We also have
\[ ((h\o\o g)S(h\o\t))h\t = ((h\o g)S(h\t\o))h\t\t = hg\]
So for all $h,g\in H$,
\[ (h\o\o(gS(h\o\t)))h\t =((h\o\o g)S(h\o\t))h\t\]
therefore $h\o(gS(h\t))=(h\o g)S(h\t)$ as required.
\endproof

So we have a notion of the left (and similarly right) adjoint action of a Hopf quasigroup $H$ when it is cocommutative and flexible.

In any Hopf quasigroup we define the associator by
 \[(hg)f=\varphi(h\o,g\o,f\o)(h\t (g\t f\t))\]

\begin{proposition} Let $H$ be a Hopf quasigroup with associator $\varphi$. Then, for all $h,g\in H$,
\begin{enumerate}
\item $\varphi$ always exists as it can be expressed as
\[ \varphi(h,g,f)=((h\o g\o) f\o)S(h\t (g\t f\t)),\quad\forall h,g,f\]
\item $\varphi(1,h,g)=\varphi(h,1,g)=\varphi(h,g,1)=\eps(h)\eps(g).1$, 
\item 
$\varphi(h\o,S(h\t),g)=\varphi(S(h\o),h\t,g)=\eps(h)\eps(g).1$\\
$\varphi(h,g\o,S(g\t))=\varphi(h,S(g\o),g\t)=\eps(h)\eps(g).1$\\
$\varphi(h\o g\o,S(g\t),S(h\t))=\varphi(S(h\o)S(g\o),g\t,h\t)=\eps(h)\eps(g).1$\\
$\varphi(S(h\o),S(g\o),g\t h\t)=\varphi(h\o,g\o,S(g\t)S(h\t))=\eps(h)\eps(g).1$\\
$\varphi(S(h\o),h\t S(g\o),g\t)=\varphi(h\o,S(h\t)g\o,S(g\t))=\eps(h)\eps(g).1$
\end{enumerate}
\end{proposition}
\proof
Part (1) uses that comultiplication is a coassociative algebra homomorphism, we have
\begin{eqnarray*}
\lefteqn{\varphi(h\o,g\o,f\o)(h\t (g\t f\t))}\\
	&	=	&	((h\o g\o) f\o)S(h\t\o (g\t\o f\t\o)))(h\t\t (g\t\t f\t\t))\\
	&	=	& ((h\o g\o) f\o)S((h\t (g\t f\t))\o))(h\t (g\t f\t))\t\\
	&	=	& ((h\o g\o) f\o)\eps(h\t (g\t f\t))\\
	&	=	& (hg)f
\end{eqnarray*}
Part (2) is straightforward, for example
\[ \varphi(1,h,g)=(h\o g\o)S(h\t g\t)=(hg)\o S((hg)\t)=\eps(h)\eps(g).1\]
Part (3) requires the identities from the definition of a Hopf quasigroup, and that $S$ is anticomultiplicative. For example,
\begin{eqnarray*}
\varphi(h\o,S(h\t),g)	&	=	& ((h\o\o S(h\t)\o) g\o)S(h\o\t (S(h\t)\t g\t))\\
	&	=	&	((h\o\o S(h\t\t))g\o)S(h\o\t (S(h\t\o) g\t))\\
	&	=	&	((h\o\o S(h\t))g\o))S(h\o\t\o (S(h\o\t\t)g\t))\\
	&	=	& ((h\o S(h\t)) g\o)S(g\t)\\
	&	=	& \eps(h)g\o S(g\t)\\
	&	=	& \eps(h)\eps(g).1
\end{eqnarray*}
The rest are similar.
\endproof

\begin{proposition} If $\Gcal$ is a quasigroup then $H=k\Gcal$ is a Hopf quasigroup with linear extension of the product and $\Delta u=u\tens u$, $\eps u = 1$, $Su=u^{-1}$ on the basis elements. 
\end{proposition}
\proof 
We check on the basis elements. The comultiplication is clearly coassociative and an algebra homomorphism. Since $\Gcal$ is a quasigroup,
\[ S(u\o)(u\t v)=u^{-1} (uv)=v = \eps(u) v \]
for all $u,v\in \Gcal$. Similarly the other identities hold for $k\Gcal$.

\endproof

Hence we have Hopf quasigroups associated to all the examples of the previous section. Conversely,  if $H$ is a Hopf quasigroup then clearly its set
\[ G(H)=\{u\in H\ |\ \Delta u=u\tens u\}\]
of grouplike elements is a quasigroup in our sense with inverse property, and Moufang if $H$ is Moufang. 

The theory clearly also includes the notion of an enveloping Hopf quasigroup $U(L)$ associated to a Mal'tsev algebra $L$ when the ground field has characteristic not $2,3$. We will recall the axioms of the required Mal'tsev bracket $[\ ,\ ]$ on $L$ where we use it in Section 6. Suffice it to say for the moment that \cite{PS} define a not necessarily associative algebra $U(L)$  as a quotient of the free nonassociative algebra on $L$ by an ideal imposing relations $xy-yx=[x,y]$ and $(x,h,g)+(h,x,g)=0$, $(h,x,g)+(h,g,x)=0$ for all $x,y\in L$ and $h,g$ in the free algebra. (Here $(a,b,c)=(ab)c-a(bc)$ is the usual additive associator on an algebra). These authors also provided a  `diagonal map'  $\Delta:U(L)\to U(L)\tens U(L)$ defined by $\Delta x=x\tens 1+1\tens x$ for all $x\in L$ extended to $U(L)$ as an algebra homomorphism. 

\begin{proposition} For $k$ not of characteristic 2,3 and $U(L)$ as above, there exists  maps $S:U(L)\to U(L)$, $\eps:U(L)\to k$ defined by $Sx=-x$ and $\eps x=0$ extended  as an antialgebra homomorphism and algebra homomorphism respectively and making $U(L)$ into a Hopf quasigroup.
\end{proposition}
\proof 
In view of Proposition~4.2 we know that if $S$ exists then it will be antimultiplicative and hence as stated. One has to verify that $S$ defined in this way is well-defined, which is clear as $S(x,h,g)=(Sg,Sh,x)$, $S(h,x,g)=(Sg,x,Sh)$ in the free non-associative algebra. It is easy to see that $S$ then obeys the properties required in Definition~4.1 in the generators. That it does so in general is easily proven by induction. Thus assume that the first identity in Definition~4.1 (say) holds on sums of products of $\le n$ elements of $L$. Now
suppose $h\in U(L)$ can be expressed as a product of $n$ such elements and let $x\in L$. Then using the (anti)-multiplicative properties of $\Delta, S$ and their form on $x$, and the definition and properties in $U(L)$ of the additive associator, and our inductive assumption, we have
\begin{eqnarray*}
\lefteqn{S((hx)\o)((hx)\t g)+S((xh)\o)((xh)\t g)}\\
	&	=	& S(h\o x)(h\t g) + S(h\o)((h\t x)g) + S(xh\o)(h\t g) + S(h\o)((xh\t)g)\\
	&	=	& -(xS(h\o))(h\t g) + S(h\o)((h\t x)g) - (S(h\o)x)(h\t g) + S(h\o)((xh\t)g)\\
	&	=	& -(x,S(h\o),h\t g)-x(S(h\o)(h\t g)) + S(h\o)(h\t,x,g) + S(h\o)(h\t(xg))\\
	&		& \quad\quad - (S(h\o),x,h\t g) - S(h\o)(x(h\t g)) + S(h\o)(x,h\t,g) + S(h\o)(x(h\t g))\\
	&	=	& (S(h\o),x,h\t g)-\eps(h)xg - S(h\o)(x,h\t,g)+\eps(h)xg \\
	&		& \quad\quad - (S(h\o),x,h\t g) - S(h\o)(x(h\t g)) + S(h\o)(x,h\t,g) + S(h\o)(x(h\t g))\\
	&	=	& 0\\
	&	=	& \eps(hx+xh)g
\end{eqnarray*}
so that this identity holds also on $hx+xh= 2hx+[x,h]$ if it holds on $h$. Now $[x,h]$ can also be expressed as a sum of products of $\le n$ generators and hence the required identity already holds by assumption on this. Hence it holds on $hx$ and hence on sums of products of $\le n+1$ elements of $L$. Similarly for the other identities in Definition~4.1. We used in the proof here that $[x,h]$ can also be expressed in terms of sums of products $\le n$ elements if $h$ can. This assertion too can be easily proven by induction. Suppose $h$ can be expressed as a product of $n$ elements and suppose that the assertion is true. Now, 
\[ [x,hy] = [x,h]y + h[x,y] + 3(x,y,h) \]
hence, by our assumption, $[x,h]y$ and $h[x,y]$ can then be written as a sum of product of $\le n+1$. We also have 
\[ 3(x,y,h) = \frac{1}{2} [[h,x],y]-\frac{1}{2}[[h,y],x]-\frac{1}{2}[h,[x,y]] \]
and so, again by our assumption, $3(x,y,h)$ can be written as a sum of products of $n$ elements. Hence, $[x,hy]$ is also a sum of products of $\le n+1$ elements of $L$ whenever $hy$ is. We used two standard identities implicit in \cite{PS}.
\endproof

\begin{proposition}
$U(L)$ as above is Moufang.
\end{proposition}
\proof 
Clearly, if $h\in L$ and $g,f\in U(L)$ then
\begin{eqnarray*}
h\o(g(h\t f))	&	=	& h(gf)+g(hf) = (hg)f-(h,g,f)+(gh)f-(g,h,f) \\
	&	=	& ((h\o g)h\t)f -(h,g,f)+(h,g,f) = ((h\o g)h\t)f 
\end{eqnarray*}
Now we prove by induction; suppose $h\in U(L)$ can be expressed as a product of $n$ elements of $L$, and that the Moufang identity with any $g,f\in U(L)$ holds for all such elements. Now, let $x\in L$ then,
\begin{eqnarray*}
(hx)\o(g((hx)\t f))	& =	& (h\o x)(g(h\t f)) + h\o(g((h\t x)f))\\
	&	=	& (h\o,x,g(h\t f)) + h\o(x(g(h\t f))) \\
	&		& \quad + h\o(g(h\t,x,f)) + h\o(g(h\t(xf))) \\
	&	=	& (h\o,x,g(h\t f)) + h\o((xg)(h\t f)) - h\o(x,g,h\t f)  \\
	&		& \quad + h\o(g(h\t,x,f)) + h\o(g(h\t(xf))) \\
	&	=	& (h\o,x,g(h\t f)) + ((h\o(xg))h\t)f - h\o(x,g,h\t f) \\
	&		& \quad + h\o(g(h\t,x,f)) + ((h\o g)h\t)(xf) \\
	&	=	& (h\o,x,g(h\t f)) + (((h\o x)g)h\t)f - ((h\o,x,g)h\t)f \\
	&		& \quad- h\o(x,g,h\t f) + h\o(g(h\t,x,f)) + (((h\o g)h\t)x)f\\
	&		& \quad - ((h\o g)h\t,x,f) \\
	&	=	& (h\o,x,g(h\t f)) + (((h\o x)g)h\t)f - ((h\o,x,g)h\t)f \\
	&		& \quad - h\o(x,g,h\t f) + h\o(g(h\t,x,f)) + (h\o g,h\t,x)f \\
	&		& \quad + ((h\o g)(h\t x))f - ((h\o g)h\t,x,f) \\
	&	=	& (((hx)\o g)(hx)\t)f +(h\o,x,g(h\t f)) - ((h\o,x,g)h\t)f \\
	&		& \quad - h\o(x,g,h\t f) + h\o(g(h\t,x,f)) + (h\o g,h\t,x)f \\
	&		& \quad - ((h\o g)h\t,x,f). \\
\end{eqnarray*}
We used that $\Delta$ is a homomorphism, the coproduct on $x$, the definition of the additive associator and its properties defining $U(L)$. It remains need to prove vanishing of
\begin{eqnarray*}
\lefteqn{(h\o,x,g(h\t f)) - ((h\o,x,g)h\t)f - h\o(x,g,h\t f)}\\
	&		& \quad\quad\quad\quad+ h\o(g(h\t,x,f)) + (h\o g,h\t,x)f - ((h\o g)h\t,x,f)
\end{eqnarray*}
To do this we again apply identities on the associator and use the Moufang identity, assumed to hold for $h$:
\begin{eqnarray*}
	&	=	& -(h\o,g(h\t f),x)-h\o(g,h\t f,x)+((h\o g)h\t,f,x)\\
	&		& \quad -((h\o,x,g)h\t)f+h\o(g(h\t,x,f))+(h\o g,h\t,x)f\\
	&	=	& -(h\o(g(h\t f)))x+h\o((g(h\t f))x)-h\o((g(h\t f))x) + h\o(g((h\t f)x))\\
	&		& \quad (((h\o g)h\t)f)x-((h\o g)h\t)(fx)-((h\o,x,g)h\t)f\\
	&		& \quad +h\o(g(h\t,x,f))+(h\o g,h\t,x)f\\
	&	=	&  -(h\o(g(h\t f)))x + h\o(g((h\t f)x))+(((h\o g)h\t)f)x-((h\o g)h\t)(fx)\\
	&		& \quad -((h\o,x,g)h\t)f+h\o(g(h\t,x,f))+(h\o g,h\t,x)f\\
	&	=	&  -(h\o(g(h\t f)))x + h\o(g((h\t f)x))+(h\o(g(h\t f)))x-((h\o g)h\t)(fx)\\
	&		& \quad -((h\o,x,g)h\t)f+h\o(g(h\t,x,f))+(h\o g,h\t,x)f\\
	&	=	& h\o(g((h\t f)x))-((h\o g)h\t)(fx)\\
	&		& \quad -((h\o,x,g)h\t)f+h\o(g(h\t,x,f))+(h\o g,h\t,x)f\\
	&	=	& h\o(g(h\t,f,x))+h\o(g(h\t(fx)))-((h\o g)h\t)(fx)\\
	&		& \quad -((h\o,x,g)h\t)f+h\o(g(h\t,x,f))+(h\o g,h\t,x)f\\	
	&	=	& h\o(g(h\t,f,x))+h\o(g(h\t(fx)))-h\o(g(h\t(fx)))\\
	&		& \quad -((h\o,x,g)h\t)f+h\o(g(h\t,x,f))+(h\o g,h\t,x)f\\
	&	=	& -h\o(g(h\t,x,f))-((h\o,x,g)h\t)f+h\o(g(h\t,x,f))+(h\o g,h\t,x)f\\
	&	=	& ((h\o,g,x)h\t)f-(h\o g,x,h\t)f\\
	&	=	& (((h\o g)x)h\t)f-((h\o(gx))h\t)f-(((h\o g)x)h\t)f+((h\o g)(xh\t))f\\
	&	=	& ((h\o g)(xh\t))f-((h\o(gx))h\t)f\\
	&	=	& ((h\o(gx))h\t)f-((h\o(gx))h\t)f\\
	&	=	& 0
\end{eqnarray*}
In the penultimate equality we have used one of the equivalent Moufang identities in view of Lemma~4.4. Hence, we have shown that the Moufang identity is satisfied for any element which can be expressed as a product of $n+1$ elements of $L$, and hence $U(L)$ is Moufang.
\endproof

Going in the other direction, if $H$ is a Moufang Hopf quasigroup with invertible antipode then the set
\[ L(H)=\{x\in H\ |\ \Delta x=x\tens 1+1\tens x\}\]
is a Mal'tsev algebra with the commutator bracket $[x,y]=xy-yx$. Indeed, from Lemma~4.4 applied to such elements we see that
\[ L(H)\subseteq N_{alt}(H)=\{x\in H\ |\ (x,h,g)=-(h,x,g)=(h,g,x)\ \forall h,g\in H\}\]
(the `alternative nucleus' used in \cite{PS}) as Mal'tsev algebras.

We can also construct examples by cross product methods as follows.

\begin{proposition}
Let $H$ be a Hopf quasigroup equipped with an action of a quasigroup $\Gcal$. Thus there is a linear action of $\Gcal$ on $H$ such that $\sigma.(hg)=(\sigma.h)(\sigma.g)$, $\sigma.1=1$, $(\sigma\tens\sigma).\Delta h=\Delta(\sigma.h)$, $\eps(\sigma.h)=\eps(h)$ and $\sigma\sigma'.h=\sigma.(\sigma'.h)$ for all $\sigma, \sigma'\in \Gcal$ and $h,g\in H$. The cross product algebra $H\rtimes k\Gcal$ is again a Hopf quasigroup.
\end{proposition}
\proof 
The (not necessarily associative) algebra product is $(h\tens\sigma)(g\tens\sigma')=h \sigma.g\tens \sigma\sigma'$ and the coproduct is the tensor product one, where $\Delta\sigma=\sigma\tens\sigma$. This is a coalgebra as required with tensor product counit. It is easy to see that $\Delta,\eps$ are algebra maps.

For the quasigroup identities, we note first a lemma that a map of the algebra and coalgebra necessarily commutes with $S$, so $\sigma.(Sh)=S(\sigma.h)$ for all $h\in H$, $\sigma\in \Gcal$.
If we consider
\[ (\sigma.h\o)(\sigma.Sh\t)=\sigma.(h\o Sh\t)=\sigma.\eps(h) = \eps(h) \]
But also,
\[ (\sigma.h\o)S(\sigma.h\t)=(\sigma.h)\o S((\sigma.h)\t)=\eps(\sigma.h)=\eps(h) \]
Hence
\[ (\sigma.h\o)S(\sigma.h\t) = (\sigma.h\o)(\sigma.Sh\t) \]
By replacing $h$ by $h\t$ and multiplying on the left by $S(\sigma.h\o)$, we can use the quasigroup identities to obtain
\[ S(\sigma.h)=\sigma.S(h)\]
as required.

Also, by the properties of $S$ proven above, we know that $S$ if it exists on $H\rtimes k\Gcal$ must be given by $S(h\tens\sigma)=(1\tens \sigma^{-1})(Sh\tens 1)=\sigma^{-1}.Sh\tens\sigma^{-1}$. We then verify the required identities in a straightforward manner. For example, using properties of a quasigroup and a Hopf quasigroup,
\begin{eqnarray*}
(h\tens\sigma)\o(S((h\tens\sigma)\t)(g\tens\sigma'))	&	=	& (h\o\tens\sigma)((\sigma^{-1}.S(h\t)\tens\sigma^{-1})(g\tens\sigma'))\\
	&	=	& (h\o\tens\sigma)((\sigma^{-1}.S(h\t))(\sigma^{-1}.g)\tens\sigma^{-1}\sigma')\\
	&	=	& (h\o\tens\sigma)(\sigma^{-1}.(S(h\t)g)\tens \sigma^{-1}\sigma')\\
	&	=	& h\o \sigma.(\sigma^{-1}.(S(h\t)g))\tens \sigma(\sigma^{-1}\sigma')\\
	&	=	& h\o (\sigma\sigma^{-1}.(S(h\t)g))\tens \sigma(\sigma^{-1}\sigma')\\
	&	=	& h\o(S(h\t)g)\tens \sigma(\sigma^{-1}\sigma')\\
	&	=	& \eps(h) g\tens \sigma'\\
	&	=	& \eps(h\tens\sigma)g\tens\sigma'
\end{eqnarray*}
\endproof

\begin{example}
Let $H=k S^{2^n-1}$ and $G=\Z_2^n$. We label the elements of the latter as $\sigma_a$ where $a\in\Z_2^n$ and note the action
\[  \sigma_a.e_b=(-1)^{a\cdot b} e_b\]
on $k_FG$. This leaves the norm $q$ invariant and hence defines an action on the set $S^{2^n-1}$. This action extends linearly to one on $H$ and leads to a cross product $k S^{2^n-1}\rtimes kZ_2^n$ which we identify with the Hopf quasigroup associated to the quasigroup cross product $S^{2^n-1}\rtimes\Z_2^n$.
\label{cross}
\end{example}

While this example is not very interesting, we will see in the next section that replacing $H$ by its dual is (put another way, we can consider cross coproducts rather than cross products).  Also, using our framework  one can extend Proposition~4.10 to an action by a general cocommutative Hopf quasigroup as we see next.  Clearly other Hopf algebra constructions can similarly be extended to the quasigroup case. 

\begin{definition}
Let $H$ be a Hopf quasigroup. A vector space $V$ is a \textit{left $H$-module} if there is a linear map $\alpha:H\tens V\to V$ written as $\alpha(h\tens v)=h.v$ such that 
\[ h.(g.v)=(hg).v, \quad 1.v=v\]
for all $h,g\in H, v\in V$. An algebra (not necessarily associative) $A$ is an \textit{$H$-module algebra} if further
\[h.(ab)=(h\o.a)(h\t.b), \quad h.1=\eps h\]
for all $h\in H, a,b\in A$. Finally, a coalgebra $C$ is an \textit{$H$-module coalgebra} if
\[\Delta(h.c)=h\o.c\o\tens h\t.c\t, \quad \eps(h.c)=\eps(h)\eps(c)\]
for all $h\in H, c\in C$.
Therefore we have the notion of a \textit{left $H$-module Hopf quasigroup}; we can similarly define right actions of Hopf quasigroups.
\end{definition}

\begin{lemma}
If $A$ is a left $H$-module algebra and a left $H$-module coalgebra, then
\[ h.S(a)=S(h.a)\]
for all $h\in H, a\in A$.
\end{lemma}
\proof
To see this, we use the definition of an action on an algebra to find
\[ (h\o. a\o)(h\t.S(a\t)) = h.(a\o S(a\t)) = \eps(h)\eps(a)\]
We also find, using the definition of an action on a coalgebra
\[ (h\o.a\o)S(h\t.a\t) = (h.a)\o S((h.a)\t) = \eps(h.a) = \eps(h)\eps(a)\]
So we have
\[ (h\o. a\o)(h\t.S(a\t)) = (h\o.a\o)S(h\t.a\t)\]
Applying this to $h\t,a\t$, multiplying on the left by $S(h\o.a\o)$ and using the Hopf quasigroup identities, gives the require identity.
\endproof

\begin{proposition}
Let $H$ be a cocommutative Hopf quasigroup and $A$ be a left $H$-module Hopf quasigroup, then there is a left cross product Hopf quasigroup $A\rtimes H$ built on $A\tens H$ with tensor product coproduct and unit and
\[ (a\tens h)(b\tens g) = a(h\o.b)\tens h\t g,\quad S(a\tens h)=S(h\t).S(a)\tens S(h\o) \]
for all $a,b\in A, g,h\in H$. 
\end{proposition}
\proof
To see that $\Delta$ is an algebra homomorphism, we compute
\begin{eqnarray*}
\Delta((a\tens h)(b\tens g))	& =	& \Delta(a(h\o.b)\tens h\t g)\\
	&	=	& a\o(h\o.b)\o\tens h\t\o g\t\tens a\t(h\o.b)\t\tens h\t\t g\t\\
	&	=	& a\o(h\o\o.b\o)\tens h\t\o g\o\tens a\t(h\o\t.b\t)\tens h\t\t g\t\\
\Delta(a\tens h)\Delta(b\tens g)	&	=	& (a\o\tens h\o)(b\o\tens g\o)\tens (a\t\tens h\t)(b\t\tens g\t)\\
	&	=	& a\o(h\o\o.b\o)\tens h\o\t g\o \tens a\t(h\t\o.b\t)\tens h\t\t g\t
\end{eqnarray*}
These are equal if $H$ is cocommutative. We will check one of the quasigroup identities, the rest are similar.
\begin{eqnarray*}
S((a\tens h)\o)((a\tens h)\t(b\tens g))	&	=	& (S(h\o)\o.S(a\o)\tens S(h\o)\t)(a\t(h\t.b)\tens h\th g)\\
	&	=	& (S(h\o)\o.S(a\o))(S(h\o)\t.(a\o(h\t\o.b)))\\
	&		& \quad \tens S(h\o)\th(h\th g)\\
	&	=	& S(h\o)\o.(S(a\o)(a\o(h\t.b)))\tens S(h\o)\t(h\th g)\\
	&	=	& \eps(a)S(h\t).(h\th.b)\tens S(h\o)(h_{(4)} g)\\
	&	=	& \eps(a) (S(h\t)h\th).b\tens S(h\o)(h_{(4)} g)\\
	&	=	& \eps(a) b\tens S(h\o)(h\t g)\\
	&	=	& \eps(a)\eps(h) b\tens g
\end{eqnarray*}
The first equality uses the definition of the comultiplication, the antipode and the multiplication. The third uses the property of an action on an algebra. The fourth equality uses the definition of $A$ a Hopf quasigroup on $a$ and the fifth uses the definition of an action on a vector space. Finally we use the Hopf quasigroup identity on $h\in H$.
\endproof

\section{Hopf coquasigroups}

Armed with the above linearised concept of a quasigroup, we can now reverse arrows on all maps. This means that in the finite dimensional case a Hopf quasigroup is equivalent on the dual linear space (by dualising all structure maps) to the following concept of a Hopf coquasigroup.

\begin{definition}
A \textit{Hopf coquasigroup} is a unital associative algebra $A$ equipped with counital algebra homomorphisms $\Delta:A\to A\tens A$, $\eps:k\to A$, and linear map $S:A\to A$ such that
\begin{equation*}(m\tens \id)(S\tens\id\tens\id)(\id\tens\Delta)\Delta=1\tens\id = (m\tens \id)(\id\tens S\tens\id)(\id\tens\Delta)\Delta \end{equation*}
\begin{equation*} (\id\tens m)(\id\tens S\tens\id)(\Delta\tens\id)\Delta = \id\tens 1 = (\id\tens m)(\id\tens\id\tens S)(\Delta\tens\id)\Delta \end{equation*}
A Hopf coquasigroup is \textit{flexible} if
\[ a\o a\t\t \tens a\t\o = a\o\o a\t \tens a\o\t \quad \forall a\in A\]
and \textit{alternative} if also
\[ a\o a\t\o \tens a\t\t = a\o\o a\o\t \tens a\t,\]
\[ a\o \tens a\t\o a\t\t = a\o\o \tens a\o\t a\t \]
for all $a\in A$. We say $A$ is \textit{Moufang} if
\[ a\o a\t\t\o \tens a\t\o \tens a\t\t\t = a\o\o\o a\o\t \tens a\o\o\t \tens a\t \quad \forall a\in A\]
\end{definition}

The term `counital' here means
\[ (\id\tens\eps)\Delta=(\eps\tens\id)\Delta=\id\]
but we do not demand that $\Delta$ is coassociative. 

\begin{proposition} Let $A$ be Hopf coquasigroup.  Then 
\begin{enumerate}
\item $m(S\tens\id)\Delta=1.\eps=m(\id\tens S)\Delta$
\item $S$ is antimultiplicatve $S(ab)=(Sb)(Sa)$ for all $a,b\in A$
\item $S$ is anticomultiplicative $\Delta(S(a))=S(a\t)\tens S(a\o)$ for all $a\in A$. 
\end{enumerate}
Hence a Hopf coquasigroup is a Hopf algebra {\em iff} it is coassociative. 
\end{proposition}
\proof
This is the dual proposition to Proposition~\ref{antipode} and the proof is obtained by writing the proof of that as diagrams, and reversing all of the arrows. We will include the proof in this dual case for clarity.
(1) is immediate from the definition. To prove (2), we consider
\[S(b\o)S(a\o)a\t\o b\t\o S(a\t\t b\t\t)\]
Using that $\Delta$, $\eps$ are algebra homomorphisms, and (1), this equals
\[ S(b\o)S(a\o)((a\t b\t)\o S((a\t b\t)\t)) = S(b\o)S(a\o)\eps(a\t b\t) = S(b)S(a)\]
We can also write the original expression as
\[ (S(b\o)(S(a\o)a\t\o)b\t\o)S(a\t\t b\t\t)\]
Using defining properties of a Hopf coquasigroup, this equals
\[ (S(b\o)b\t\o)S(ab\t\t)=S(ab)\]
This holds for all $a,b\in A$, hence $S$ is antimultiplicative.

To prove (3), we consider $S(a\t\o\o)a\t\o\t\o S(a\t\t)\o \tens S(a\o)a\t\o\t\t S(a\t\t)\t$. On one hand this equals,
\[ S(a\t\o\o)(a\t\o\t S(a\t\t))\o \tens S(a\o)(a\t\o\t S(a\t\t))\t \]
We apply the identity $a\o\o\tens a\o\t S(a\t) = a\tens 1$ to $a\t$ to obtain,
\[ S(a\t)1\o \tens S(a\o)1\t = S(a\t)\tens S(a\o)\]
On the other hand, we have
\[ S(a\t\o\o)a\t\o\t\o S(a\t\t)\o \tens S(a\o)a\t\o\t\t S(a\t\t)\t \]
and we apply the identity $S(a\o)a\t\o \tens a\t\t = 1\tens a$ to $a\t\o$ to obtain,
\[ S(a\t\t)\o\tens S(a\o)a\t\o S(a\t\t)\t \]
Now we apply the same identity to $a\t$ to obtain
\[ S(a)\o\tens  S(a)\t \]
Therefore
\[ S(a)\o\tens  S(a)\t = S(a\t)\tens S(a\o)\]
for all $a\in A$, as required.
\endproof

\begin{proposition}
Let $A$ be a Hopf coquasigroup, then $S^2=\id$ if $A$ is commutative or cocommutative.
\end{proposition}
\proof
Let $a\in A$. Since $S$ is antimultiplicative, we have
\[ S^2(a)=S^2(a\o)S(a\t\o)a\t\t=S(a\t\o S(a\o))a\t\t \]
Now, if $A$ is commutative, this equals
\[ S(S(a\o)a\t\o)a\t\t=S(1)a = a \]
by definition of a Hopf coquasigroup. Alternatively, if $A$ is cocommutative, it equals
\[ S(a\o\o S(a\t))a\o\t = S(a\o\t S(a\t))a\o\o = S(1)a = a \]
In either case, $S^2=\id$ as required.
\endproof

\begin{lemma}
Let $A$ be a Hopf coquasigroup such that $S^{-1}$ exists, then the following identities are equivalent for all $a\in A$
\begin{enumerate}
\item $a\o a\t\t\o \tens a\t\o \tens a\t\t\t = a\o\o\o a\o\t \tens a\o\o\t \tens a\t$
\item $a\o\o\o \tens a\o\o\t a\t \tens a\o\t = a\o \tens a\t\o a\t\t\t \tens a\t\t\o$
\item $a\o\o a\t\t \tens a\o\t \tens a\t\o = a\o\o a\t \tens a\o\t\o \tens a\o\t\t$
\end{enumerate}
\end{lemma}
\proof
This lemma is dual to lemma \ref{Moufang} and the proof can be obtained by writing the proof of \ref{Moufang} as diagrams, and reversing the arrows. We will demonstrate one part here.

Assume (3) holds. By applying this to $a\o$, we find
\begin{eqnarray*}
\lefteqn{a\o\o\o a\o\t\t S(a\t)\t\tens a\o\o\t \tens a\o\t\o S(a\t)\o}\\
	&	=	& a\o\o\o a\o\t S(a\t)\t\tens a\o\o\t\o \tens a\o\o\t\t S(a\t)\o
\end{eqnarray*}
Using that $\Delta$ is an algebra homomorphism and the property that $a\o\o \tens a\o\t S(a\t)$, the LHS equals
\[ a\o\o\o(a\o\t S(a\t)\t \tens a\o\o\t \tens (a\o\t S(a\t))\o = a\o \tens a\t \tens 1\]
So we obtain
\[ a\o\tens a\t\tens 1 = a\o\o\o a\o\t S(a\t\o)\tens a\o\o\t\o \tens a\o\o\t\t S(a\t\t)\]
By applying $S^{-1}$ to the third component, then $\Delta$ and multiplying, we find
\[ a\o a\t\t\o\tens a\t\o \tens a\t\t\t = a\o\o\o a\o\t S(a\t\o)a\t\t\o \tens a\o\o\t \tens a\t\t\t\]
Now we can apply the definition of a Hopf coquasigroup on $a\t$ on the RHS to get
\[ a\o a\t\t\o\tens a\t\o \tens a\t\t\t = a\o\o\o a\o\t \tens a\o\o\t \tens a\t\]
which is identity (1).
\endproof

\begin{lemma}
Let $A$ be a commutative flexible Hopf coquasigroup, then for all $a\in A$,
\[ a\o S(a\t\t) \tens a\t\o = a\o\o S(a\t)\tens a\o\t \]
\end{lemma}
\proof
By commutativity and flexibility we have
\[ a\o\o S(a\o\t\t)a\t \tens a\o\t\o = a\o S(a\t\o\t)a\t\t\tens a\t\o\o\]
which, equals $a\o\tens a\t$  by definition on $a\t$. But also, by properties of a Hopf coquasigroup,
\[ a\o\o\o S(a\o\t)a\t \tens a\o\o\t = a\o\tens a\t\]
Therefore
\[ a\o\o S(a\o\t\t)a\t \tens a\o\t\o = a\o\o\o S(a\o\t)a\t \tens a\o\o\t\]
By applying this to $a\o$ and multiplying on the right by $S(a\t)$, we obtain
\[ a\o S(a\t\t) \tens a\t\o = a\o\o S(a\t)\tens a\o\t\]
as required.
\endproof

So when $A$ is a commutative flexible Hopf coquasigroup we have a notion of the left (and similarly right) adjoint coaction.

As in the previous theory, we define the coassociator, now by
 \[ (\Delta\tens\id)\Delta(a)=\Phi(a\o)(\id\tens\Delta)\Delta(a\t)\]
For the next proposition, we will use some convenient notation; let $A$ be a Hopf coquasigroup, and $a\in A$. We write $\Phi(a)=\Phi_a\bo\tens\Phi_a\bt\tens\Phi_a\bth$

\begin{proposition} Let $A$ be a Hopf coquasigroup with coassociator $\Phi$. Then, for all $a\in A$,
\begin{enumerate}
\item $\Phi$ always exists as it can be expressed as
\[ \Phi(a)= a\o\o\o S(a\t)\o \tens a\o\o\t S(a\t)\t\o \tens a\o\t S(a\t)\t\t\]
\item $(\eps\tens\id\tens\id)\Phi_a=(\id\tens\eps\tens\id)\Phi_a=(\id\tens\id\tens\eps)\Phi_a=\eps(a).1$
\item $\Phi_a\bo S(\Phi_a\bt)\tens \Phi_a\bth = S(\Phi_a\bo)\Phi_a\bt\tens \Phi_a\bth$\\
$=\Phi_a\bo\tens S(\Phi_a\bt)\Phi_a\bth = \Phi_a\bo\tens \Phi_a\bt S(\Phi_a\bth) $\\
$=\Phi_a\bo\o S(\Phi_a\bth)\tens \Phi_a\bo\t S(\Phi_a\bt) = S(\Phi_a\bo\o)\Phi_a\bth\tens S(\Phi_a\bo\t)\Phi_a\bt$\\
$=S(\Phi_a\bo)\Phi_a\bth\t\tens S(\Phi_a\bt)\Phi_a\bth\o = \Phi_a\bo S(\Phi_a\bth\t)\tens \Phi_a\bt S(\Phi_a\bth\o)$\\
$=S(\Phi_a\bo)\Phi_a\bt\o\tens S(\Phi_a\bt\t)\Phi_a\bth = \Phi_a\bo S(\Phi_a\bt\o)\tens \Phi_a\bt\t S(\Phi_a\bth) = \eps(a).1$
\end{enumerate}
\end{proposition}
\proof
This proof is dual to one in the previous section. We will prove (1) and give an example of (3).
Since $\Delta$ is an algebra homomorphism and using properties of a Hopf coquasigroup, we have
\begin{eqnarray*}
\Phi(a\o)(\id\tens\Delta)\Delta(a\t)	&	=	& a\o\o\o\o S(a\o\t)\o a\t\o \tens a\o\o\o\t S(a\o\t)\t\o a\t\t\o\\
	&		&	\tens a\o\o\t S(a\o\t)\t\t a\t\t\t\\
	&	=	& a\o\o\o\o (S(a\o\t)a\t)\o\tens a\o\o\o\t (S(a\o\t)a\t)\t\o \\
	&		& \tens a\o\o\t (S(a\o\t)a\t)\t\t\\
	&	=	& a\o\o\tens a\o\t \tens a\t\\
	&	=	& (\Delta\tens\id)\Delta(a)
\end{eqnarray*}
We will prove that $\Phi_a\bo S(\Phi_a\bt)\tens \Phi_a\bth=\eps(a).1$
\begin{eqnarray*}
\Phi_a\bo S(\Phi_a\bt)\tens \Phi_a\bth	&	=	& a\o\o\o S(a\t)\o S(a\o\o\t S(a\t)\t\o)\tens a\o\t S(a\t)\t\t \\
	&	=	&	a\o\o\o S(a\t)\o S(S(a\t)\t\o)S(a\o\o\t) \tens a\o\t S(a\t)\t\t\\
	&	=	& a\o\o\o S(a\o\o\t)\tens a\o\t S(a\t)\\
	&	=	& \eps(a\o\o) \tens a\o\t S(a\t)\\
	&	=	& 1\tens a\o S(a\t)\\
	&	=	& \eps(a).1
\end{eqnarray*}
Where we have used that $\Delta$ is an algebra homomorphism, $S$ is anti-multiplicative, and the identities of a Hopf coquasigroup. The rest are similar.
\endproof

Clearly as $\Gcal_n$ is a finite quasigroup, $k\Gcal_n$ is a finite-dimensional Hopf quasigroup algebra and hence its dual $k[\Gcal_n]$ of functions on $\Gcal_n$ with pointwise multiplication is a Hopf coquasigroup. However, the dual theory is more powerful and also allows `coordinate algebra' versions of infinite-dimensional quasigroups $\Gcal$ as we now demonstrate. Specifically, we consider the algebra $k[S^{2^n-1}]$ of functions on the spheres $S^{2^n-1}\subset k_FG$, where $G=\Z_2^n$. This algebra is generated by the functions
\[ x_a(u)=u_a\]
that pick out the value of the $a$-th coordinate of $u=\sum_a u_a e_a$. More precisely, $k[S^{2^n-1}]$ is defined to be the (commutative) polynomial algebra $k[x_a:a\in\Z_2^n]$ with relations $\sum_a x_a^2=1$.

\begin{proposition} $A=k[S^{2^n-1}]$ is Hopf coquasigroup  with coproduct $\Delta x_c=\sum_{a+b=c}x_a\tens x_b F(a,b)$, counit $\eps x_a=\delta_{a,0}$, and antipode $Sx_a=x_aF(a,a)$. The coassociator is
\begin{eqnarray*}
\Phi(x_d)	& = & \sum_{a+b+c+a'+b'+c'=d} x_ax_{a'}\tens x_bx_{b'}\tens x_cx_{c'}F(a',a')F(b',b')F(c',c')\\
	&		& F(a,b)F(a+b,c)F(c',b')F(b'+c',a')F(a+b+c,a'+b'+c')
\end{eqnarray*}
 \end{proposition}
\proof
By construction, it is clear that $\Phi$ satisfies $(\Delta\tens\id)\Delta(x_d)=\Phi(x_d)(\id\tens\Delta)\Delta(x_d)$.
To show that $k[S^{2^n-1}]$ is a Hopf coquasigroup, we will prove only that $(m\tens \id)(S\tens\id\tens\id)(\id\tens\Delta)\Delta=1\tens\id$, the others are similar. The LHS give us
\[ \sum_{a+b+c=d} x_ax_b\tens x_c F(a,a)\phi(a,b,c)F(a+b,c)F(a,b)\]
Consider the case when $a=b$ so $c=d$, this gives
\[ \sum_a x_a^2\tens x_d F(a,a)\phi(a,a,d)F(0,d)F(a,a)=\sum_a x_a^2\tens x_d = 1\tens x_d\]
Now consider the case when $a\ne b$. We claim that the term with given values for $a$ and $b$ cancels with the term with $a'=b, b'=a$. These give, respectively,
\[ x_ax_b\tens x_c F(a,a)\phi(a,b,c)F(a+b,c)F(a,b)\]
\[ x_bx_a\tens x_c F(b,b)\phi(b,a,c)F(b+a,c)F(b,a)\]
When $a=0$ and hence $b\ne 0$, these become
\[ x_0x_b\tens x_c F(b,c)\]
\[ x_bx_0\tens x_c F(b,b)F(b,c)=-x_0x_b\tens x_c F(b,c)\]
which cancel. When $a,b\ne 0$, these become
\[ -x_ax_b\tens x_c \phi(a,b,c)F(a+b,c)F(a,b)\]
\[ -x_bx_a\tens x_c \phi(b,a,c)F(b+a,c)F(b,a)=x_ax_b\tens x_c \phi(a,b,c)F(a+b,c)F(a,b)\]
which also clearly cancel. Hence,
\[ \sum_{a+b+c=d} x_ax_b\tens x_c F(a,a)\phi(a,b,c)F(a+b,c)F(a,b)=1\tens x_d\]
as required. \endproof

It is not obvious and a nice check that all the other properties above follow, eg $S$ is antimultiplicative etc.

\begin{proposition} $k[S^{2^n-1}]$ is a Moufang Hopf coquasigroup, and hence flexible and alternative.
\end{proposition}
\proof Consider a generator $x_f\in k[S^{2^n-1}]$ and the two expressions
\[ (1)\quad  x_f\o x_f\t\t\o\tens x_f\t\o\tens x_f\t\t\t = \sum_{a+b+c+d=f} x_ax_c\tens x_b\tens x_d  F(c,d)F(b,c+d)F(a,b+c+d) \]
\[ (2)\quad  x_f\o\o\o x_f\o\t\tens x_f\o\o\t\tens x_f\t = \sum_{a+b+c+d=f} x_ax_c\tens x_b\tens x_d F(a,b)F(a+b,c)F(a+b+c,d). \]
Then $k[S^{2^n-1}]$ is Moufang if (1)=(2). We will consider the different possible cases.

\textbf{Case 1:} $f=0$\\
Since $a+b+c+d=0$, $\phi(a+b,c,d)=\phi(a,b,c+d)=1$, so we have
\begin{eqnarray*}
(2) & =	& \sum_{a+b+c+d=0} x_ax_c\tens x_b\tens x_d F(a,b)F(a+b,c)F(a+b+c,d) \\
	&	=	& \sum_{a+b+c+d=0} x_ax_c\tens x_b\tens x_d \phi(a+b,c,d)\phi(a,b,c+d) F(c,d)F(b,c+d)F(a,b+c+d) \\
	&	=	& \sum_{a+b+c+d=f} x_ax_c\tens x_b\tens x_d  F(c,d)F(b,c+d)F(a,b+c+d)\\
	&	=	& (1)
\end{eqnarray*}

\textbf{Case 2:} $f\ne 0$\\
\textbf{Part expression:} we look at the terms where $a=0$ and $c=0$ in the sums. We claim that if $a=0$ in (1), this equals the sum when $c=0$ in (2). Assume $b,c,d$ linearly independent so that $R(a,b)=\phi(a,b,c)=-1$, and consider $c=0$ in (2);
\begin{eqnarray*}
(2)	&	=	& \sum_{a+b+d=f} x_ax_0\tens x_b\tens x_d F(a,b)F(a+b,d)\\
	&	=	& \sum_{b+a+d=f} x_0x_a\tens x_b\tens x_d F(a,b)F(b+a,d)\\
	&	=	& -\sum_{b+a+d=f} x_0x_a\tens x_b\tens x_d F(b,a)F(b+a,d)\\
	&	=	& -\sum_{b+c+d=f} x_0x_{c}\tens x_{b}\tens x_{d} F(b,c)F(b+c,d) \quad \text{by relabeling}\\
	&	=	& -\sum_{b+c+d=f} x_0x_{c}\tens x_{b}\tens x_{d}\phi(b,c,d)F(c,d)F(b,c+d)\\
	&	=	& \sum_{b+c+d=f} x_0x_{c}\tens x_{b}\tens x_{d}F(c,d)F(b,c+d)\\
	&	=	& (1) \quad \text{with $a=0$}
\end{eqnarray*}

Similarly, $c=0$ in (1) equals the sum when $a=0$ in (2).

\textbf{Part expression:} terms where $b=0$. We have 
\[ (1)  = \sum_{a+c+d=f} x_ax_c\tens x_0\tens x_d  F(c,d)F(a,c+d) \]
\[ (2)  = \sum_{a+c+d=f} x_ax_c\tens x_0\tens x_d F(a,c)F(a+c,d) \]
If $a,c,d$ are linearly dependent, then $\phi(a,c,d)=1$, so clearly (1)=(2). Suppose $a,c,d$ are linearly independent, then using commutativity of the generators
\begin{eqnarray*}
(2)	&	=	& \sum_{a+c+d=f} x_ax_c\tens x_0\tens x_d F(a,c)F(a+c,d)\\
	&	=	& - \sum_{c+a+d=f} x_cx_a\tens x_0\tens x_d F(c,a)F(c+a,d)\\
	&	=	&	- \sum_{a+c+d=f} x_ax_c\tens x_0\tens x_d F(a,c)F(a+c,d)\\
	&	=	& - \sum_{a+c+d=f} x_ax_c\tens x_0\tens x_d \phi(a,c,d)F(c,d)F(a,c+d) \\
	&	=	& \sum_{a+c+d=f} x_ax_c\tens x_0\tens x_dF(c,d)F(a,c+d)\\
	&	=	& (1)
\end{eqnarray*}
The second equality uses $R(a,c)=-1$ since $a,c$ linearly independent, the third equality comes from relabeling, and the fifth uses linear independence of $a,c,d$. Relabeling gives (1).

\textbf{Part expression:} terms where $d=0$. We have
\[ (1) = \sum_{a+b+c=f} x_ax_c\tens x_b\tens x_0  F(b,c)F(a,b+c) \]
\[ (2) = \sum_{a+b+c=f} x_ax_c\tens x_b\tens x_0  F(a,b)F(a+b,c) \]
If $a,b,c$ are linearly dependent then $\phi(a,b,c)=1$, so clearly (1)=(2). Suppose $a,b,c$ are linearly independent, then $R(a,b)=R(a+b,c)=-1$ so we have,
\begin{eqnarray*}
(2)	&	=	& \sum_{a+b+c=f} x_ax_c\tens x_b\tens x_0  F(a,b)F(a+b,c)\\
	&	=	& \sum_{c+b+a=f} x_cx_a\tens x_b\tens x_0 F(b,a)F(c,b+a)\\
	&	=	& \sum_{a+b+c=f} x_ax_c\tens x_b\tens x_0 F(b,c)F(a,b+c)\\
	&	=	& (1)
\end{eqnarray*}
Relabeling gives (1).

\textbf{Part expression:} terms where $a+b=0,a+c=0,a+d=0,b+c=0,b+d=0,c+d=0$. We have
\[ (1) = \sum_{a+b+c+d=f} x_ax_c\tens x_b\tens x_d  F(c,d)F(b,c+d)F(a,b+c+d) \]
\[ (2) = \sum_{a+b+c+d=f} x_ax_c\tens x_b\tens x_d \phi(a+b,c,d)\phi(a,b,c+d)F(c,d)F(b,c+d)F(a,b+c+d) \]
If any two of the variables are equal, $\phi(a+b,c,d)=\phi(a,b,c+d)=1$ by the symmetric property of $\phi$ and linear dependence of the variables. Hence (1)=(2) in each of these cases.

\textbf{Part expression:} terms where $a+b+c=0$. 
Note that this means $d=f$, and we assume $a,b,c\ne 0$ as these cases are done. We also assume$a,b,c\ne f$ and $a+c+f\ne 0$ as these each imply $a+d,b+d,c+d=0$, which we have already covered.
\[ (1) = \sum_{a,c} x_ax_c\tens x_{a+c}\tens x_f F(c,f)F(a+c,c+f)F(a,a+f) \]
\[ (2) = \sum_{a,c} x_ax_c\tens x_{a+c}\tens x_f F(a,a+c)F(c,c) = \sum_{a,c} x_ax_c\tens x_{a+c}\tens x_f F(a,c) \]
Now, using $\phi(a,c,f)=-1$ and $R(a,c)=-1$, we find
\begin{eqnarray*}
(1)	&	=	& \sum_{a,c} x_ax_c\tens x_{a+c}\tens x_f F(c,f)F(a+c,c+f)F(a,a+f)\\
	&	=	& \sum_{a,c} x_ax_c\tens x_{a+c}\tens x_f F(c,c+f)F(a+c,c+f)F(a,f)\\
	&	=	& \sum_{a,c} x_ax_c\tens x_{a+c}\tens x_f \phi(a,c,c+f)F(a,c)\\
	&	=	& -\sum_{a,c} x_ax_c\tens x_{a+c}\tens x_f F(a,c)\\
	&	=	& -\sum_{c,a} x_cx_a\tens x_{c+a}\tens x_f F(a,c)\\
	&	=	& -\sum_{a,c} x_ax_c\tens x_{a+c}\tens x_f F(c,a) \quad \text{by relabeling}\\
	&	=	& \sum_{a,c} x_ax_c\tens x_{a+c}\tens x_f F(a,c)\\
	&	=	& (2)
\end{eqnarray*}

\textbf{Part expression:} terms where $a+c+d=0$. 
As in the previous subcase, $b=f\ne0$, and we can assume $a,c,d\ne0$ and that $a,c,f$ are linearly independent. Using $R(a,f)=-1$ we find
\begin{eqnarray*}
(1)	&	=	& \sum_{a,c} x_ax_c\tens x_f\tens x_{a+c} F(c,a+c)F(f,a)F(a,f+a)\\
	&	=	& \sum_{a,c} x_ax_c\tens x_f\tens x_{a+c} F(c,a)F(f,a)F(a,f)\\
	&	=	& \sum_{a,c} x_ax_c\tens x_f\tens x_{a+c} F(c,a)R(a,f)\\
	&	=	& -\sum_{a,c} x_ax_c\tens x_f\tens x_{a+c} F(c,a)
\end{eqnarray*}
Now, using $\phi(f,a,c)=-1$ since $a,f,c$ are linearly independent, and relabeling we find
\begin{eqnarray*}
(2)	&	=	& \sum_{a,c}x_ax_c\tens x_f\tens x_{a+c} F(a,f)F(a+f,c)F(a+f+c,a+c)\\
	&	=	& \sum_{a,c}x_ax_c\tens x_f\tens x_{a+c} F(f,a)F(f+a,c)F(f,a+c)\\
	&	=	& \sum_{a,c}x_ax_c\tens x_f\tens x_{a+c} \phi(f,a,c)F(a,c)\\
	&	=	& -\sum_{a,c}x_ax_c\tens x_f\tens x_{a+c} F(a,c)\\
	&	=	& -\sum_{c,a}x_cx_a\tens x_f\tens x_{c+a} F(a,c)\\
	&	=	& -\sum_{a,c}x_ax_c\tens x_f\tens x_{a+c} F(c,a) \quad \text{by relabeling}\\
	&	=	& (1)
\end{eqnarray*}

\textbf{Part expression:} terms where $a+b+d=0$ and $b+c+d=0$.  We claim that the sum with $a+b+d=0$ in (1) equals the sum with $b+c+d=0$ in (2). Let us consider (2) with $d=b+c$ and so $a=f$. We can assume $b+c\ne 0$ and $f,b,c$ are linearly independent, as these cases have been proven.
\begin{eqnarray*}
(2)	&	=	& \sum_{b,c} x_fx_c\tens x_b\tens x_{b+c} F(f,b)F(f+b,c)F(f+b+c,b+c)\\
	&	=	& -\sum_{b,c} x_fx_c\tens x_b\tens x_{b+c} F(f,b)F(f+b,c)F(f,b+c)\\
	&	=	& -\sum_{b,c} x_fx_c\tens x_b\tens x_{b+c} \phi(f,b,c)F(b,c)\\
	&	=	& \sum_{b,c} x_fx_c\tens x_b\tens x_{b+c} F(b,c)\\
	&	=	& \sum_{a,b} x_fx_a\tens x_b\tens x_{a+b} F(b,a)\quad \text{by relabeling}\\
	&	=	& -\sum_{a,b} x_ax_f\tens x_b\tens x_{a+b} \phi(b,a,f)F(b,a)\\
	&	=	& -\sum_{a,b} x_ax_f\tens x_b\tens x_{a+b} F(b+a,f)F(b,a+f)F(a,f)\\
	&	=	& \sum_{a,b} x_ax_f\tens x_b\tens x_{a+b} F(f,b+a)F(b,f+a)F(a,f)\\
	&	=	& \sum_{a,b} x_ax_f\tens x_b\tens x_{a+b} F(f,b+a)F(b,f+a+b)F(a,f+a)\\
	&	=	& (1) \quad \text{with $a+b+d=0,c=f$}
\end{eqnarray*}

Similarly, the case with $b+c+d=0$ in (1) equals the case with $a+b+d=0$ in (2).

\textbf{Part expression:} Terms where $a,b,c,d$ are linearly independent.  
In this case $\phi(a+b,c,d)=\phi(a,b,c+d)=-1$, hence
\begin{eqnarray*}
(2) &	=	& \sum_{a+b+c+d=f} x_ax_c\tens x_b\tens x_d F(a,b)F(a+b,c)F(a+b+c,d)\\
	&	=	& \sum_{a+b+c+d=f} x_ax_c\tens x_b\tens x_d\phi(a+b,c,d)\phi(a,b,c+d)F(c,d)F(b,c+d)F(a,b+c+d)\\
	&	=	& \sum_{a+b+c+d=f} x_ax_c\tens x_b\tens x_dF(c,d)F(b,c+d)F(a,b+c+d)\\
	&	=	& (1)
\end{eqnarray*}
\endproof

\begin{remark} Although $\Phi$ has rather a large number of terms we can differentiate it at the identity $(1,0,0,0,0,0,0,0)$ by setting $x_ix_j=0$ if $i,j>0$ to obtain
\[ \Phi_*(x_d) = \sum_{a+b+c=d} (\phi(a,b,c)-1)F(b,c)F(a,b+c) x_a\tens x_b\tens x_c. \]
This is essentially adjoint to the map $(x,y,z)=(xy)z-x(yz)$ on the underlying quasialgebra $k_FG$, which on basis elements takes the form
\[ (e_a,e_b,e_c)=(\phi(a,b,c)-1)F(b,c)F(a,b+c)e_{a+b+c}\]
\end{remark}

\begin{proposition}
Let $A$ be a Hopf coquasigroup equipped with an action of a group $G$. Thus there is a linear action of $G$ on $A$ such that $\sigma.(ab)=(\sigma.a)(\sigma.b)$, $\sigma.1=1$, $(\sigma\tens\sigma).\Delta a=\Delta(\sigma.a)$, $\eps(\sigma.a)=\eps(a)$ and $(\sigma\sigma').a=\sigma.(\sigma'.a)$ for all $\sigma,\sigma'\in G$ and $a,b\in A$. The cross product algebra $A\rtimes G$ is again a Hopf coquasigroup and is Moufang if $A$ is.
\end{proposition}
\proof 
The algebra is a standard cross product construction to give an associative algebra $A\rtimes kG$. We define the tensor product $\Delta,\eps$ and check that these are algebra homomorphisms just as in the Hopf quasigroup case. We then straightforwardly verify the coquasigroup identities, for example,
\begin{eqnarray*}
S((a\tens\sigma)\o)(a\tens\sigma)\t\o\tens(a\tens\sigma)\t\t	&	=	& S(a\o\tens\sigma)(a\t\o\tens\sigma)\tens(a\t\t\tens\sigma)\\
	&	=	& (\sigma^{-1}.S(a\o)\tens\sigma^{-1})(a\t\o\tens\sigma)\tens(a\t\t\tens\sigma)\\
	&	=	& (\sigma^{-1}.S(a\o))(\sigma^{-1}.a\t\o)\tens\sigma^{-1}\sigma\tens(a\t\t\tens\sigma)\\
	&	=	& (\sigma^{-1}.(S(a\o)a\t\o)\tens 1)\tens(a\t\t\tens\sigma)\\
	&	=	& (1\tens 1)\tens(a\tens\sigma)
\end{eqnarray*}
Finally, let $A$ be a Moufang Hopf coquasigroup, and $G$ be a group acting on $A$, then $A\rtimes kG$ is a Moufang Hopf coquasigroup. This is straightforward to see; we consider
\[ (a\tens\sigma)\o(a\tens\sigma)\t\t\o \tens (a\tens\sigma)\t\o \tens (a\tens\sigma)\t\t\t\]
By definition of the coproduct, this equals
\[ (a\o\tens\sigma)(a\t\t\o\tens\sigma)\tens (a\t\o\tens\sigma)\tens(a\t\t\t\tens\sigma)\]
Since $A$ is Moufang, and again using the definition of the coproduct, we see this equals
\[ (a\tens\sigma)\o\o\o(a\tens\sigma)\o\t \tens (a\tens\sigma)\o\o\t \tens (a\tens\sigma)\t\]
and hence $A\rtimes kG$ is Moufang.
\endproof

\begin{example}
The Hopf coquasigroup $k[S^{2^n-1}]$ has an action of $G=\Z_2^n$ by
\[ \sigma_a.x_b=(-1)^{a\cdot b} x_b\]
and the resulting cross product $k[S^{2^n-1}]\rtimes \Z_2^n$ is a noncommutative Hopf coquasigroup.
\end{example}

The action is adjoint to the one in Example \ref{cross} in the last section. One can verify all the required properties. The resulting associative algebra has generators $x_a$ and  $\sigma_{001},\sigma_{010}$ and $\sigma_{100}$ with relations
\[ x_a\sigma_{001}=(-1)^{a_1}\sigma_{001}x_a,\quad x_a\sigma_{010}=(-1)^{a_2}\sigma_{010}x_a,\quad x_a\sigma_{100}=(-1)^{a_3}\sigma_{100}x_a.\]
This is therefore the first example of a `quantum' Hopf coquasigroup, which we believe to be the first of many.

Clearly Proposition~5.10 hints at further general constructions at the level of Hopf coquasigroups. Here we limit ourselves, for completeness, to  the dual constructions to those for Hopf quasigroups at the end of the last section.

\begin{definition}
Let $A$ be a Hopf coquasigroup. A vector space $V$ is a \textit{right $A$-comodule} if there is a linear map $\beta:V\to V\tens A$ written as $\beta(v)=v\bo\tens v\bt$ such that 
\[ v\bo\bo\tens v\bo\bt\tens v\bt = v\bo\tens v\bt\o\tens v\bt\t, \quad v\bo\eps(v\bt)=v\]
for all $v\in V$. An algebra $H$ is an \textit{$A$-comodule algebra} if further
\[\beta(hg)=\beta(h)\beta(g), \quad \beta(1)=1\tens 1\]
for all $h,g\in H$. Finally, a coalgebra  (not necessarily coassociative) $C$ is an \textit{$A$-comodule coalgebra} if
\[c\bo\o\tens c\bo\t\tens c\bt=c\o\bo\tens c\t\bo\tens c\o\bt c\t\bt, \quad \eps(c\bo)c\bt=\eps(c)\]
for all $c\in C$.
Therefore we have the notion of a \textit{right $A$-comodule Hopf coquasigroup}; we can similarly define left actions of Hopf coquasigroups.
\end{definition}

\begin{lemma}
If $C$ is a right $H$-comodule algebra and a right $A$-comodule coalgebra, then the coaction commutes with the antipode, that is
\[ S(c\bo)\tens c\bt = S(c)\bo\tens S(c)\bt \]
for all $c\in C$.
\end{lemma}
\proof
Using the property of the coaction on an algebra we have
\[ c\o\bo S(c\t)\bo\tens c\o\bt S(c\t)\bt = (c\o S(c\t))\bo\tens (c\o S(c\t)\bt = \eps(c).1 \]
Now using the property of a coaction on a coalgebra we obtain
\[ c\o\bo S(c\t\bo)\tens c\o\bt c\t\bt = c\bo\o S(c\bo\t)\tens c\bt = \eps(c\bo)c\bt = \eps(c).1\]
So we have
\[ c\o\bo S(c\t)\bo\tens c\o\bt S(c\t)\bt = c\o\bo S(c\t\bo)\tens c\o\bt c\t\bt\]
from which we can obtain the required identity.
\endproof

\begin{proposition}
Let $A$ be a Hopf coquasigroup and let $C$ be a right $A$-comodule Hopf coquasigroup. There is a right cross coproduct  Hopf coquasigroup $A\rcocross C$ built on $A\tens C$ with tensor product algebra and counit and
\[ \Delta(a\tens c)=a\o\tens c\o\bo\tens a\t c\o\bt\tens c\t\]
\[ S(a\tens c)=S(ac\bt)\tens S(c\bo)\]
for all $a\in A,c\in C$. 
\end{proposition}
\proof
We check that $\Delta$ is an algebra homomorphism in the same way as for Hopf quasigroups; next we check the coquasigroup identities. Now we are ready to check the coquasigroup identities.
\begin{eqnarray*}
\lefteqn{S((a\tens c)\o)(a\tens c)\t\o\tens (a\tens c)\t\t = }\\
	&	=	& (S(a\o c\o\bo\bt)\tens S(c\o\bo\bo))(a\t\o c\o\bt\o\tens c\t\o\bo)\tens (a\t\t c\o\bt\t c\t\o\bt\tens c\t\t)\\
	&	=	& (S(c\o\bo\bt)S(a\o)a\t\o c\o\bt\o\tens S(c\o\bo\bo)c\t\o\bo) \tens (a\t\t c\o\bt\t c\t\o\bt\tens c\t\t)\\
	&	=	& (S(c\o\bo\bt)c\o\bt\o\tens S(c\o\bo\bo)c\t\o\bo) \tens (ac\o\bt\t c\t\o\bt\tens c\t\t)\\
	&	=	& (S(c\o\bt\o)c\o\bt\t\o\tens S(c\o\bo)c\t\o\bo)\tens (ac\o\bt\t\t c\t\o\bt\tens c\t\t)\\
	&	=	& (1\tens S(c\o\bo)c\t\o\bo)\tens (ac\o\bt c\t\o\bt\tens c\t\t)\\
	&	=	& (1\tens S(c\o)\bo c\t\o\bo)\tens (a S(c\o)\bt c\t\o\bt\tens c\t\t)\\
	&	=	& (1\tens (S(c\o)c\t\o)\bo)\tens (a(S(c\o)c\t\o)\bt\tens c\t\t)\\
	&	=	& (1\tens 1)\tens (a\tens c)
\end{eqnarray*}
The fourth equality uses the definition of a coaction on a vector space on $c\o$. The fifth and seventh equalities use the coquasigroup identities on $c\o\bt$ and $c$ respectively. The sixth equality uses the property of the antipode commuting with the coaction. The other identities are similar.
\endproof

\section{Differential Calculus on Hopf Coquasigroups and $k[S^7]$}

Let $A$ be an associative algebra. As usual, we define an $A$-module and an $A$-bimodule in the usual way, with commuting left and right actions written multiplicatively. As in most approaches to noncommutative geometry we define differential structures by specifying the bimodule of 1-forms.

\begin{definition} A \textit{first order differential calculus} over $A$ is a pair $(\Omega^1,\extd)$ such that\\
1) $\Omega^1$ is an $A$-bimodule\\
2) $\extd:A\to A$ is a linear map satisfying
\[ \extd(ab)=\extd(a)b+a\extd(b)\]
3) $\Omega^1={\rm span}\{a\extd b\, |\, a,b\in A\}$
\end{definition}

 The universal calculus, $\Omega^1_{univ}$ is defined in the usual way as the kernel of the multiplication map with $\extd a=1\tens a-a\tens 1$. When $A$ is a Hopf coquasigroup, the algebra structure is an associative algebra and so the above definition  makes sense, however we would like some form of `translation invariance' with respect to the quasigroup multiplication expressed in the coproduct. Left invariance  is effected for ordinary Hopf algebras by the `left Maurer-Cartan form' $\omega$ and in the Hopf coquasigroup case we take this as the definition. 

\begin{definition}
Let $A$ be a Hopf coquasigroup. A first order differential calculus $\Omega^1$ over $A$ is \textit{left covariant} if it is a free left $A$-module over $\im(\omega)$, i.e.
\[ \Omega^1=A.\im(\omega),\]
where $\omega:A^+\to\Omega^1$ is defined by
\[ \omega(a)=(Sa\o)\extd a\t \]
\end{definition}

We can extend the definition of $\omega$ to $A$ by $a=\eps(a)+(a-\eps(a))$ (the counit projection) and $\omega(1)=0$. Then clearly, $a\o\omega(a\t)=a\o Sa\t\o \extd a\t\t=\extd a$ by the Hopf coquasigroup identities allows us to  recover  the calculus from knowledge of the Maurer-Cartan form. We similarly define  $\Omega^1$ to be  \textit{right covariant} if 
\[ \Omega^1=\im(\omega_R).A, \quad  \omega_R(a) =( \extd a\o)S a\t \]
with respect to a right-handed Maurer Cartan form. The calculus is bicovariant if both of these hold. The universal calculus is bicovariant.

\begin{lemma} Let $A$ be a Hopf coquasigroup and $A^+=\rm{ker}\eps$ the augmentation ideal. Then $\Omega^1_{univ}\iso A\tens A^+$ via the left Maurer-Cartan form. 
\end{lemma}

\proof
Define $r:\Omega^1_{univ}\to A\tens A^+$ by $a\tens b\mapsto ab\o\tens b\t$. We can check that the RHS lies in $A\tens A^+$ by applying $(\id\tens\eps)$.
\[(\id\tens\eps)r(a\tens b)=(\id\tens\eps)(ab\o\tens b\t)=ab\o\eps(b\t)=ab=0\]
since $a\tens b\in\Omega^1_{univ}$. Hence $r(a\tens b)\in A\tens A^+$. The inverse map is $r^{-1}(a\tens b)=aS(b\o)\tens b\t=a\omega(b)$ provided by the left Maurer-Cartan form for the universal calculus. One can also see directly that it lies in $\Omega^1_{univ}$ by
\[mr^{-1}(a\tens b)=m(aS(b\o)\tens b\t)=aS(b\o)b\t=a\eps(b)=0\]
since $b\in A^+$. To show these maps are mutually inverse we require the defining properties of a Hopf coquasigroup; we find,
\begin{eqnarray*}
rr^{-1}(a\tens b)	&=	&r(aS(b\o)\tens b\t) =	 aS(b\o)b\t\o\tens b\t\t =a\tens b
\\
r^{-1}r(a\tens b)	&=	&r^{-1}(ab\o\tens b\t) =	ab\o S(b\t\o)\tens b\t\t 
									=	a\tens b.
\end{eqnarray*}
using the Hopf coquasigroup identities.
\endproof

\begin{theorem}
Let $A$ be a Hopf coquasigroup. Left covariant first order calculi over $A$ are in 1-1 correspondence with right ideals $I\subset A^+$ of $A$.
\end{theorem}
\proof
Let $\Omega^1$ be a left covariant first order differential calculi over $A$, then $\Omega^1=A.\im\omega$, where $\omega:A^+\to \Omega^1$ is defined as
\[ \omega(a)=S(a\o)\extd(a\t) \]
for each $a\in A^+$. Define $I=\ker\omega$. Then $I$ is a right ideal of $A$: let $x\in I,a\in A$, then
\begin{eqnarray*}
\omega(xa)	&	=	& S(x\o a\o)\extd(x\t a\t)\\
	&	=	& S(a\o)S(x\o)x\t \extd(a\t) + S(a\o)S(x\o)\extd(x\t)a\t \\
	&	=	& S(a\o)\eps(x) \extd(a\t) + S(a\o)\omega(x)a\t\\
	&	=	& 0
\end{eqnarray*}
since $x\in I\subset A^+$.
We note that $A.\im(\omega) \iso A\tens \im(\omega)$ by the product, and $\im(\omega)\iso A^+/I$, hence $\Omega^1\iso A\tens A^+/I$.

Conversely, given a right ideal $I'\subset A^+$ of $A$, define $N=r^{-1}(A\tens I')$. Then $N$ is a sub-bimodule of $\Omega^1_{univ}$: let $a,x\in A$ and $b\in I'$
\begin{eqnarray*}
r(r^{-1}(a\tens b)(1\tens x))	&	=	& r(r^{-1}(a\tens b))\Delta(b)\\
	&	=	& (a\tens b)\Delta(b) \quad\in A\tens I'
\end{eqnarray*}
since $I'$ is a right ideal. Therefore $r^{-1}(a\tens b)(1\tens x)\in N$. Similarly,
\begin{eqnarray*}
r((x\tens 1)r^{-1}(a\tens b))	&	=	& (x\tens 1)r(r^{-1}(a\tens b)\\
	&	=	& (x\tens 1)(a\tens b) \quad \in  A\tens I'
\end{eqnarray*}
so $(x\tens 1)r^{-1}(a\tens b)\in N$, as required. Therefore $\Omega^1=\Omega^1_{univ}/N$ is a FODC over $A$. It remains to check that $\Omega^1$ is left covariant. By the previous lemma, $\Omega^1_{univ}\iso A\tens A^+$, and there are canonical projections $\Omega^1_{univ}\to\Omega^1_{univ}/N$ and $A\tens A^+\to A\tens A^+/I'$. Therefore, we have an isomorphism $A\tens A^+/I'\iso\Omega^1$ sending $a\tens b\to a\omega(b)$, where $\omega$ is the left Maurer-Cartan form on $\Omega^1$. Thus $\Omega^1\iso A\tens A^+/I'\iso A\tens\im(\omega)$, and $\Omega^1$ is left covariant.

These processes are mutually inverse; let $\pi:\Omega^1_{univ}\to\Omega^1$ be the canonical projection sending $a\tens b\to a\extd b$, then $\omega(x)=\pi r^{-1}(1\tens x)$ and it is clear that $I'=I$.
\endproof

Similarly, right covariant first order calculi over $A$ are in 1-1 correspondence with left ideals in $A^+$. Bicovariant calculi correspond to a compatible pair of ideals or to right ideals (say) with further properties.  Also, any calculus in degree 1 can be extended to higher degree although not uniquely. Here $\extd$ extends as a graded derivation with $\extd^2=0$. 

\begin{proposition}\label{mceqn} Let $\Omega^1$ be a left-covariant differential on a Hopf coquasigroup $A$ and $\Omega^2$ any extension to degree 2. The left Maurer-Cartan form obeys
\[ \extd \omega(a)+(Sa\o)a\t\o\o \omega(a\t\o\t)\omega(a\t\t)=0,\quad \forall a\in A^+.\]
\end{proposition}
\proof We note that $\extd 1=0$ and hence
\begin{eqnarray*} 0&=& \extd 1= (\extd((Sa\o)a\t\o))Sa\t\t\o\, \extd a\t\t\t\\
&=& (\extd S a\o)a\t\o Sa\t\t\o\, \extd a\t\t\t+Sa\o (\extd a\t\o) Sa\t\t\o\ \extd a\t\t\t \\
&=& (\extd S a\o) \extd a\t+Sa\o (\extd a\t\o) Sa\t\t\o\ \extd a\t\t\t
\end{eqnarray*}
using the Hopf coquasigroup identities, the derivation property of $\extd$ and the Hopf coquasigroup identities again. Now the first term is $\extd\omega(a)$ by the Leibniz rule and $\extd^2=0$ while the second term is $S a\o \extd a\t\o \omega(a\t\t)=Sa\o a\t\o\o \omega(a\t\o\t)\omega(a\t\t)$ by the remark after Definition~6.2. \endproof

\subsection{Left-invariant vector fields on $S^7$ and the Lie algebra $g_2$}

Applying the results above, we conclude with some first remarks on  `lie algebra' objects which should be associated to $S^{2^n-1}$ as some kind of `group' and should be defined by vector fields associated to the left Maurer-Cartan form obtained above. We work with $k$ of characteristic not 2,3 for convenience and $S^{2^n-1}$ is defined by $F$ giving a composition algebra as in Section~2.

We will use the convention that $x_a$ indicates $a\in \{0,1,\cdots, n-1\}$ while $x_i$ indicates $i\in\{1,\cdots,n-1\}$. We use the classical calculus on $k[S^{2^n-1}]$ defined by $I=(k[S^{2^n-1}]^+)^2$, the square of the augmentation ideal. We have  $\omega(x_0-1)=\sum_a (Sx_a)\extd x_a F(a,a)=\sum_a x_a\extd x_a=0$ and can further deduce that 
\[ \omega(x_ix_j)=\omega((x_0-1)^2)=0,\quad \omega(x_0x_i)=\omega(x_i)=:\omega_i\]
and similarly for higher degree monomials. We see that $\{\omega_i\}$ defines a basis of $\Lambda^1$.  Since this is a classical calculus we also have an exterior algebra with $\omega_i$ anti-commuting among themselves.

\begin{definition} We define the `left-invariant' vector fields $\del^i$ on $S^{2^n-1}$ and `structure functions' $c_i{}^{jk}$  for $i,j,k\in \Z_2^n$, $i,j,k\ne 0$, by $c_i{}^{jk}=-c_i{}^{kj}$ and 
\[\extd f= \sum_i \del^i(f)\omega(x_i),\quad \extd \omega_i+\sum_{j,k}c_i{}^{jk}\omega_j\omega_k=0.\]
\end{definition}

The Maurer-Cartan form also provides a `Lie bracket' on $\Lambda^{1*}$ as adjoint to a generalised `Lie cobacket' $\delta:\Lambda^1\to \Lambda^1\tens \Lambda^1$ defined by projection of the adjoint coaction. This is a method which works for quantum groups and in our case we similarly apply the `adjoint coaction' on $k[S^{2^n-1}]$ and hence on $k[S^{2^n-1}]^+$ (defined by the map in Lemma~5.5 since $k[S^{2^n-1}]$ is commutative and flexible). We then project this `coaction' to a map $k[S^{2^n-1}]^+\to k[S^{2^n-1}]^+\tens k[S^{2^n-1}]^+$ via the counit projection $\id-1\eps$ and then project further via $\omega$ to find $\delta$. A short computation gives
\[ \delta \omega_k=2\sum_{i+j=k}\omega_i\tens \omega_j F(i,j)\]
as a natural `Lie coalgebra' structure on $\Lambda^1$. Note that this is antisymmetric in output due to the form of $\Rcal$ and can also be obtained from the  Maurer-Cartan equations via the structure functions $c_i{}^{jk}$ evaluated at the quasigroup identity (see Lemma~\ref{leftvec}). Dually, we obtain some kind of `Lie algebra' on $\Lambda^{1*}$ taken with basis $\{\del^i\}$ and 
\[ [[\del^i,\del^i]]:=0,\quad [[\del^i,\del^j]]:=2F(i,j)\del^{i+j},\quad i\ne j.\]
This bracket  reproduces the `vector cross product' on $S^7$ when $n=3$ but is obtained now geometrically as a `Lie-type' structure on the algebraic tangent bundle at the identity. As for the theory of analytic Moufang quasigroups, one knows in this case that it is in fact a Mal'tsev algebra. In our setting this follows easily from properties of $F$:
 
 \begin{proposition} For any $F$ obeying the composition algebra identities in Section~2, the bracket $[[\ ,\ ]]$  on the tangent bundle of $k[S^{2^n-1}]$ is a Mal'tsev algebra i.e.
 \[ [[ J(x,y,z),x]]=J(x,y,[[x,z]]);\quad J(x,y,z)=[[x,[[y,z]]\, ]]+ [[y,[[z,x]]\, ]]+[[z,[[x,y]]\, ]]. \] 
 \end{proposition}
 \proof Working in our basis and the generic case of $i,j,k$ distinct and $i\ne j+k$ we have
 \[ J(\del^i,\del^j,\del^k)=8\left(F(j,k)F(i,j+k)+F(k,i)F(j,k+i)+F(i,j)F(k,i+j)\right)\del^{i+j+k}\]
 and the identity we need to show is 
 \begin{eqnarray*}
 &&\left(F(j,k)F(i,j+k)+F(k,i)F(j,k+i)+F(i,j)F(k,i+j)\right)F(i+j+k,i)\\
 &&=F(i,k)\left(F(j,i+k)F(i,i+j+k)+F(i+k,i)F(j,k)+F(i+k,j)F(i+k,i+j)\right)\end{eqnarray*}
 which using $F(i+j+k,i)=-F(j+k,i)=F(i,j+k)$ etc., and cancelling terms reduces to the 
 identity $ F(i,j+k)F(k,i+j)=F(i,k)F(i+k,i+j)$ which holds on comparison with $\phi(i,k,i+j)=\phi(i,k,j)=-1$. The degenerate cases are easily handled separately. \endproof
 
 This result is the differential analogue of Proposition~5.8 that $k[S^{2^n-1}]$ is Moufang. The Mal'tsev bracket is not, however, the commutator of vector fields unless the associator $\phi$ is trivial, as the next lemma shows. 

\begin{lemma} \label{leftvec} The vector fields $\del^i$ are given on $k[S^{2^n-1}]$ by
\[ \del^i=-\sum_a F(a,i)x_{i+a}{\del\over\del x_a}\]
and have commutator 
\[ [\del^i,\del^j]=2F(i,j)\left(-\sum_a F(a,i+j)\phi(a,i,j) x_{i+j+a}{\del\over\del x_a}\right),\quad \forall i\ne j.\]
The structure functions are given by
\[ c_i{}^{jk}=-F(i,j)F(i+j,k)\sum_a\phi(a,i,j)\phi(a,i+j,k)F(a,i+j+k)x_ax_{a+i+j+k}\]
for all $i,j,k\ne0$, are totally antisymmetric in the three indices, and obey
\[\eps (c_i{}^{jk})=F(j,k)\delta_{i,j+k}.\]
\end{lemma}
\begin{proof} We have $\extd f=f\o\omega(f\t)$ as explained after Definition~6.2, which on generators means $d x_a=\sum_{b+c=i}F(b,c)x_b\omega(x_c)=\sum_i F(i+a,i)x_{i+a}\omega_i$ as $\omega(x_0)=0$. Comparing with the definition of $\del^i$ and noting that $F(i+a,i)=-F(a,i)$ verifies our formula  on the generators. Since the calculus in the present case is commutative (the $x_i$ commute with differentials), the Leibniz rule for $d$ implies that the $\del^i$ are derivations. Hence they take the form shown on all of $k[S^{2^n-1}]$. The commutator of two such vector fields is: 
\begin{eqnarray*} && \sum_{a,b}[x_{i+a}F(a,i){\del\over\del x_a}, x_{i+b}F(b,i){\del\over\del x_b}]\\
&=&\sum_b x_{i+j+b}F(j+b,i) F(b,j){\del\over\del x_b}-\sum x_{j+i+a}F(a,i)F(i+a,j){\del\over x_a}\\
&=&\sum_a \left(F(a+j,i)F(a,j)-F(a,i)F(a+i,j)\right)x_{i+j+a}{\del\over\del x_a}\end{eqnarray*}
which gives the answer stated on using the definition of $\phi$ in each of the two products of $F$. We changed variables from $b$ to $a$ in one of the sums. 

For the computation of the structure functions we read off the iterated coproduct and antipode in the Maurer-Cartan equation in Proposition~\ref{mceqn} as
\begin{eqnarray*} 
\extd\omega_i&=&-\sum_{a+b+j+k=i}F(a,a)x_ax_b\omega_j\omega_k\, F(a,b+j+k)F(b+j,k)F(b,j)\\ 
&=&-\sum_a x_a x_{a+i+j+k} \omega_j\omega_k F(a,i)F(a+i+k,k)F(a+i+j+k,j)\\ 
&=&-\sum_a x_a x_{a+i+j+k}\omega_j\omega_k F(a,i)F(a+i,k)F(a+i+k,j)\\
&=&-\sum_a x_a x_{a+i+j+k}\omega_j\omega_k F(a,i)F(a+i,k+j)F(k,j)\phi(a+i,k,j)\\
&=&-F(k,j)F(i,k+j)\sum_a x_a x_{a+i+j+k}\omega_j\omega_k \phi(a+i,k,j)\phi(a,i,k+j)F(a,i+j+k)\\
\end{eqnarray*}
on substituting $b=a+i+j+k$ and identities for $F$ from Section~2. We used $\phi=\del F$ for the last two equalities. From this we see that
\[ c_i{}^{jk}=-F(j,k)F(i,j+k)\sum_a\phi(a,i,j+k)\phi(a+i,j,k)F(a,i+j+k)x_ax_{a+i+j+k}\]
when $j\ne k$, which is manifestly antisymmetric in $j,k$. When $j=k$ we have
\[ c_i{}^{jj}=\sum_aF(a,i)x_ax_{a+i}=\sum_a F(a,a)F(a,a+i)x_ax_{a+i}=\eps(x_i)=0\]
as also required. We then use the cocycle identity for $\phi$ and $\phi=\del F$ to cast $c_i{}^{jk}$ in the equivalent form stated. In this second form we see similarly that $c_i{}^{jk}$ is antisymmetric in $i,j$ and vanishes when $i=j$. The application of $\eps$ is immediate in the first form of $c_i{}^{jk}$ as only $a=0$ and $i=j+k$ can contribute. \end{proof}

If $n=2$ we obtain the left invariant vector fields on $k[S^3]$ and in this case $\phi=1$ and we have a closed Lie algebra with $[\del^i,\del^j]=[[\del^i,\del^j]]$, i.e. the commutator of the vector fields represents the Lie algebra. The Lie algebra in this case is $su_2$ and we see that its structure constants are rather simply $2F(i,j)$ in our basis labelled by $0\ne i,j\in \Z_2^2$. We also see precisely how the lack of closure in the case of $S^7$ depends only on the cocycle $\phi$.  Similarly, from the structure functions we see that the impact of nontrivial $\phi$ is that these depend on $x$. However, this $x$-dependence is quite mild and one can show (we have done this by direct computation with Mathematica) that
\[ \sum_{l,k}c_l{}^{ik}c_k{}^{jl}=-(2^n-2)\delta_{ij}\]
in all cases $n\le 3$. Moreover, we see after dualising that the value of $2c_i{}^{jk}$ at the identity of $S^7$ (or everywhere for $n<3$) defines the same  Mal'tsev or Lie algebra $[[\ ,\ ]]$. 

We can look similarly at right-invariant vector fields on $k[S^{2^n-1}]$. These are defined by $\extd f=\omega_R(x_i)\del^i_R(f)$ and similar computation gives 
\[ \del_i^R=-\sum_a F(i,a)x_{i+a}{\del\over\del x_a}.\]
On a group manifold the left and right invariant vector fields commute, while in the nonassociative
case they do not as another (more conventional) expression of the nonassociativity.  Paralleling standard results in the theory of analytic Moufang loops we have
\begin{corollary}
On $k[S^{2^n-1}]$ we have the identity
\[[\del^i,\del^j]-[[\del^i,\del^j]]=-2 [\del^i,\del^j_R]\]
\end{corollary}
\proof By a similar calculation as for Lemma~6.8 we have
\begin{eqnarray*} -2[\del^i,\del^j_R]&=&-2\sum_a\left(F(a+j,i)F(j,a)-F(a,i)F(j,a+i)\right)x_{i+j+a}{\del\over x_a}\\
&=&-2\sum_a (1-\phi(a,i,j))F(a+j,i)F(j,a)x_{i+j+a}{\del\over\del x_a}\\
&=&-2\sum_a(1-\phi(a,i,j))F(j,i)F(a,i+j)\phi(a,j,i)x_{i+j+a}{\del\over\del x_a}\\
&=&2F(i,j)\sum_a(\phi(a,i,j)-1)(-F(a,i+j)x_{i+j+a}{\del\over\del x_a})=[\del^i,\del^j]-[[\del^i,\del^j]]\end{eqnarray*}
when $i\ne j$ and 0 when $i=j$. We used Lemma~6.8 to recognise the result. \endproof

This expresses the failure of $[[\ ,\ ]]$ defined on the space of left-invariant vector fields to be represented by their commutator. There is a similar result for the commutator of right-invariant vector fields compared to the Mal'tsev bracket  $[[\del^i_R,\del^j_R]]:=2F(i,j)\del_R^{i+j}$ of two right-invariant vector fields.

Finally, let us apply all our tools to obtain the structure constants of $g_2$. The latter is the Lie algebra of derivations on the octonions and it is known that these can all be obtained  in the form
\[ D_{x,y}z=[[x,y],z]- 3((xy)z-x(yz)),\quad x,y,z\in \O.\]
It suffices here to consider $x,y\ne 1$, and one similarly obtains derivations for any alternative algebra. In our description of the octonions etc as quasialgebras $k_FG$, we take $D_{ij}=D_{e_i,e_j}$ for  imaginary basis elements and acting by the same formula on the algebra.

\begin{proposition} For $F$ giving a composition algebra the operations $D_{ij}$ on $k_FG$  take the form
\[ D_{ij}(e_a)=F(i,j) F(i+j,a)\psi(i,j,a)e_{i+j+a};\quad \psi(i,j,a)=3\phi(i,j,a)-1-2\Rcal(i+j,a)\]
where
\[ \psi(i,j,a)=\begin{cases} 0 & a=0, i+j \ {\rm or}\ i=j\cr 4 & a=i, j \cr -2 & i,j,a\ {\rm linearly\ independent}\end{cases}.\]
The derivation property of such $D_{ij}$ corresponds to the identity
\[ \psi(i,j,a+b)=\phi(i+j,a,b)\psi(i,j,a)+\Rcal(i+j,a)\psi(i,j,b).\]
\end{proposition}
\proof When $i\ne j$ we use the definition of the product of $k_FG$ to find
\begin{eqnarray*}&&\kern-10pt  [[e_i,e_j],e_a]-3((e_ie_j)e_a-e_a(e_je_a))\\
&=&\left(2 F(i,j)F(i+j,a)(1-\Rcal(a,i+j))-3(1-\phi(i,j.a))F(i,j)F(i+j,a)\right)e_{i+j+a}\end{eqnarray*}
which simplifies as stated in terms of a function $\psi$.  The values of $\psi$ follow from Lemma~2.1 and from this we see that the formula for $D_{ij}$ also applies as zero when $i=j$. Finally, we know from \cite{Ma99} that $k_FG$ under our assumptions is alternative, hence the $D_{ij}$ are derivations. However, $D_{ij}(e_a\cdot e_b)=(D_{ij}(e_a)\cdot e_b+e_a\cdot D_{ij}e_b$ translates in terms of the stated formulae to
\begin{eqnarray*}&& \kern-30pt F(a,b)F(i+j,a+b)\psi(i,j,a+b) \\
&=& F(i+j+a,b)F(i+j,a)\psi(i,j,a)+F(a,i+j+b)F(i+j,b)\psi(i,j,b).\end{eqnarray*}
We divide through by $F(a,b)F(i+j,a+b)$ to recognise $\phi(i+j,a,b)$ in the first term,  $\phi(i+j,a,b)\phi(a,j+j,b)=1$ in the second, and $\Rcal$ as stated. The resulting identity can also be seen to hold by case analysis of  the values of $\psi,\phi,\Rcal$ (or by using Mathematica).
\endproof
 
 We now dualise these operations to obtain vector fields on $k_FG$ as characterised by $\langle D_{ij}x_a,e_b\rangle=-\langle x_a,D_{ij}e_b\rangle$ and extended as derivations.
 
\begin{lemma} \label{Dij} 
The induced vector fields on $k_FG$ take the form
\[ D_{ij}=F(i,j)\sum_{a} F(a,i+j)\psi(i,j,a)x_{i+j+a}{\del\over\del x_a}.\]
and descend to vector fields on $S^{2^n-1}$. As such we have
\[D_{ij}=-{3\over 2}[\del^i,\del^j] + {1\over 2} [[\del^i,\del^j]]+ [[\del^i_R,\del^j_R]]\]
\end{lemma}
\proof The form on the generators follows immediately, and is then extended as derivations on products of the $x_a$.  Note that $-F(i+j,a)=F(a,i+j)$ given that only $a\ne 0,i+j$ contribute in view of Proposition~6.10. We have to show that these vector fields  vanish on $\sum_a x_ax_a$ so that they actually define vector fields on $S^{2^n-1}$. Indeed, for $i\ne j$,
\[ D_{ij}(\sum_b x_b^2)=2F(i,j)\sum_a F(a,i+j)\psi(i,j,a)x_{i+j+a}x_a.\]
Under a change of variables $a\to i+j+a$ we have $\psi(i,j,i+j+a)=\psi(i,j,a)$ from the values of $\psi$ in Proposition~6.10 (each case is invariant), the quadratic in $x$ is unchanged but $F(a+i+j,i+j)=-F(a,i+j)$ from the identities in Section~2. Hence the sum changes sign, hence vanishes. Finally, looking at the vector fields $D_{ij}$ on $S^{2^n-1}$ we see that the terms from $3\phi(i,j,a)-1$ in $\psi(i,j,a)$ give $-(3/2)[\del^i,\del^j]$ and $(1/2)[[\del^i,\del^j]]$ respectively, from Lemma~6.8. The $-2\Rcal(i+j,a)$ term in $\psi(i,j,a)$ by contrast converts $F(a,i+j)$ into $F(i+j,a)$ and  hence similarly gives $[[\del_R^i,\del_R^j]]$. \endproof

It is also possible to unpack the formula for $D_{ij}$ in Lemma~6.11 more explicitly in terms of the usual infinitesimal action of the rotation group in 8 dimensions (or `orbital angular momentum'):

\begin{corollary} We can write $D_{ij}$ explicitly as $D_{ii}=0$ and
\[D_{ij} =4(x_i {\del\over\del x_j}-x_j{\del\over\del x_i})-2F(i,j)\sum_{k\ne i,j, i+j}F(k,i+j)x_{i+j+k}{\del\over\del x_k},\quad \forall i\ne j\]
\end{corollary}
\proof Note that $\psi(i,j,a)$ vanishes when $a=0$ or $a=i+j$, and when we exclude these then $\Rcal(a,i+j)=-1$ if $i\ne j$. Hence we can also write when $i\ne j$ that
\[ D_{ij}=F(i,j)\sum_{k\ne i+j}F(k,i+j)\left(3\phi(i,j,k)+1\right)x_{i+j+k}{\del\over\del x_k}\] 
We then obtain the stated result  on further splitting off the $k=i,j$ cases (in the sum $i,j,k$ are then linearly independent over $\Z_2$). In passing, we note that the sum is most of $\del^{i+j}$  and hence we can also write 
\[ D_{ij}-[[\del^i,\del^j]]=6(x_i {\del\over\del x_j}-x_j{\del\over\del x_i})+2F(i,j)(x_{i+j}{\del\over\del x_0}-x_0{\del\over\del x_{i+j}})\]
provided $i\ne j$. 
 \endproof
 
\begin{proposition} The vector fields $D_{ij}$ on $S^{2^n-1}$ obey
\begin{enumerate}
\item $D_{ij}=-D_{ji}$ for all $i,j$.
\item $F(i,j)D_{i+j,k}+F(j,k)D_{j+k,i}+F(k,i)D_{k+i,j}=0$ for all distinct $ i,j,k$
\item When $n=3$, $ \sum_{i+j=k}F(i,j)D_{ij}=0$ for all $k$.
\end{enumerate}
\end{proposition}
\proof Part (1) is immediate as $F(i,j)=-F(j,i)$ when $i\ne j$ and $\psi(i,j,a)=0$ when $i=j$. 

We note that  part (2) is actually true for all $i,j$ provided we replace $F(i,j)$ by the actual structure constants of $[[\ ,\ ]]$ as given by $F(i,j)-F(j,i)$. Then if $i=j\ne k$ (say), the first term is zero and the 2nd and third terms cancel. When all three are distinct as stated, we split the proof into two cases. When they are linearly dependent, so $i+j+k=0$, all terms are zero as $D_{ii}=D_{jj}=D_{kk}=0$ by part (1) so in this case part (2) is empty. When they are linearly independent then $F(i,j)F(i+j,k)=-F(i,j+k)F(j,k)=F(j,k)F(j+k,i)$ since $\phi(i,j,k)=-1$ and $i\ne j+k$. Hence this expression is cyclically invariant under the change $i\to j\to k\to i$. Thus, obtaining $D_{i+j,k}$ etc. from the formula stated at the start of the proof of Corollary~6.12, we have 
\begin{eqnarray*} F(i,j)F(i+j,k)\sum_{a\ne 0,i+j+k}&& \left(3\phi(i+j,k,a)+3\phi(j+k,i,a)+3\phi(k+i,j,a)+3\right)\\
&&\quad F(a,i+j+k)x_{i+j+k+a}{\del\over\del x_a}\end{eqnarray*}
for the expression in part (1). The easiest way to see that the expression in brackets is zero is to note that for $n\le 3$ the vector $a$   cannot be linearly independent of  $i,j,k$. Hence we must have one of  $a=i,j,k,i+j,j+k,k+i$. Then, for example, if $a=k$ we have $3+\phi(j+k,i,k)+3\phi(k+i,j,k)+3=3+3\phi(j,i,k)+3\phi(i,j,k)+3=0$. 

For part (3) we change the order of summation and write $j=i+k$, so the expression of interest is
\[ \sum_{a\ne 0,k}\left(\sum_{i\ne k}\psi(i,j,a)\right)F(a,k)x_{k+a}{\del\over\del x_a}=(16-2^{n+1})\sum_{a\ne 0,k}F(a,k)x_{k+a}{\del\over\del x_a}\]
because in the sum over $i$ there are $2^n-2$ values of $i$ (as we exclude 0,$k$) and of these  $i=a,a+k$ each give $\psi(i,k+i,a)=4$ according to the values in Proposition~6.10 and the rest have $i,k+i,a$ linearly independent and hence each give $-2$. Hence the sum is $8-2(2^n-4)=16-2^{n+1}$. This vanishes  when $n=3$ (otherwise it is proportional to $x_k{\del\over\del x_0}-x_0{\del\over\del x_k}-\del^k$ in view of Lemma~6.8). \endproof

For $n=3$ the 3rd set of relations include the 2nd set, hence in this case there are at most  (7.6/2)-7=14 independent vector fields and (at least over $\R$) there are exactly this many as the derivations we started with are known to span all the derivations of the octonions and hence the 14-dimensional Lie algebra $g_2$. We conclude that the commutators of these $D_{ij}$ among themselves must give the structure constants of $g_2$ in terms of the data $F,\phi,\Rcal$ on $\Z_2^3$. For $n=2$ the 2nd set of relations is empty and the 3rd does not apply; indeed $D_{ij}=4(x_i{\del\over\del x_j}-x_j{\del\over\del x_i})$ from Corollary~6.12 and generate the Lie algebra $so(3)$ of rotations of the 3-sphere in this case.

The combinations in Lemma~6.11 of left and right invariant vector fields occur in the theory of analytic Moufang loops and in that context they are called `Yamagutian' vector fields \cite{Pal}  (more precisely,  $D_{ij}=-3Y(\del^i,\del^i)$ in terms of the notation there). Parts (1)(2) of Proposition~6.13 are our algebraic version of known properties of this, while part (3) appears to be new. We can similarly compute the `Yamaguti bracket'\cite{Yam}
\[ [x,y,z]:=[[x,[[y,z]]\, ]]+[[y,[[z,x]]\, ]]-[[z,[[x,y]]\, ]]\]
 in terms of which the commutator of two Yamagutians takes the form
\[ 6[Y(x,y),Y(z,w)]=Y([x,y,z],w)+Y(z,[x,y,w])\]
as explained in \cite{Pal}.  In our algebraic setting we obtain:
 
 \begin{theorem}  Our vector fields on $S^{2^n-1}$ obey
 \[ [D_{ij},\del^k]=-{1\over 2}[\del^i,\del^j,\del^k]=F(i,j)F(k,i+j)\psi(i,j,k)\del^{i+j+k}\] 
\[ [D_{ij},D_{kl}]=F(i,j)F(k,i+j)\psi(i,j,k)D_{i+j+k,l}-F(i,j)F(l,i+j)\psi(i,j,l)D_{i+j+l,k}\]
We obtain the structure constants of a Lie algebra spanned by the $\{D_{ij}\}$ in terms of the structure constants of $k_FG$.
\end{theorem}
\proof We first prove the formula for $[\del^i,\del^j,\del^k]$. If $i=j=k$ then this is zero since in each term there is a zero first application of $[[\ ,\ ]]$ and the right hand side is also zero as $\psi=0$ in this case. Similarly if $k=j\ne i$. If $k=i+j$ then each term of the Yamagutian vanishes due to a zero second application of $[[\ ,\ ]]$ and indeed $\psi=0$ in this case also. If $i=j\ne k$ then the last term is zero and the first two terms cancel so the Yamagutian is again zero. The right hand side has $\psi(i,i,k)=0$ from Proposition~6.10. If $i=k\ne j$ the 2nd term in the Yamagutian is zero and the first and last coincide, so $-{1\over 2}[\del^i,\del^j,\del^i]=[[\del^i,[[\del^i,\del^j]]\, ]]=4F(i,j)F(i,i+j)\del^j$. This agrees with $\psi(i,j,i)=4$ from Proposition~6.10.  Finally, if $i,j,k$ are linearly independent then 
\[ [\del^i,\del^j,\del^k]=(4F(j,k)F(i,j+k)+4F(k,i)F(j,k+i)-4F(i,j)F(k,i+j))\del^{i+j+k}\]
and the first two terms in the bracketted expression coincide with each other and with $4F(i,j)F(k,i+j)$ (as in the proof of Proposition~6.12) hence $-{1\over 2}[\del^i,\del^j,\del^k]=-2 F(i,j)F(k,i+j)\del^{i+j+k}$ as required since $\psi(i,j,k)=-2$ in this case. 

Next, we compute
\begin{eqnarray*}[D_{ij},\del^k]&=&-F(i,j)\sum_a F(a,i+j)\psi(i,j,a)F(b,k)[x_{i+j+a}{\del\over\del x_a},x_{k+b}{\del\over\del x_b}]\\
&=&-F(i,j)\sum_b F(k+b,i+j)\psi(i,j,k+b) F(b,k)x_{i+j+k+b}{\del\over\del x_b}\\
&&+F(i,j)\sum_a F(a,i+j)\psi(i,j,a)F(i+j+a,k)x_{i+j+k+a}{\del\over\del x_a}\\
&=&-F(i,j)\sum_a\left( F(a+k,i+j)F(a,k)\psi(i,j,a+k)-F(a,i+j)F(a+i+j,k)\psi(i,j,a)\right)\\
&&\quad\quad\quad\quad\quad\quad\quad\quad\quad\quad\quad\quad\quad\quad\quad x_{i+j+k+a}{\del\over\del x_a}\\
&=& - F(i,j)F(k,i+j)\sum_a \phi(a,i+j,k) \left(\psi(i,j,a+k)-\Rcal(k,i+j)\psi(i,j,a)\right) \\
&&\quad\quad\quad\quad\quad\quad\quad\quad\quad\quad\quad\quad  F(a,i+j+k)x_{i+j+k+a}{\del\over\del x_a}
\end{eqnarray*}
where we combined the sums in the 3rd equality by change of variables from $b$ to $a$. We then used the definitions of $\phi(a,k,i+j)=\phi(a,i+j,k)$ to obtain $F(a,i+j+k)F(k,i+j)$ for each quadratic of $F$ in the sum. We then identify the combination of $\phi,\psi,\Rcal$ in the sum as $\psi(i,j,k)$ by the derivation property of $\psi$ in Propostion~6.10, giving $F(i,j)F(k,i+j)\psi(i,j,k)\del^{i+j+k}$ as required.

Our formula for the Yamaguti bracket then leads to the corresponding result for $[D_{ij},D_{kl}]$ as stated. To verify this in our algebraic framework we compute the commutators 
\begin{eqnarray*}&&\kern-20pt [D_{ij},D_{kl}]=F(i,j)F(k,l)\sum_{a,b}F(a,i+j)F(b,k+l)\psi(i,j,a)\psi(k,l,b)[x_{i+j+a}{\del\over\del x_a},x_{k+l+b}{\del \over\del x_b}]\\
&=&F(i,j)F(k,l)\sum_a (F(k+l+a,i+j)F(a,k+l)\psi(i,j,k+l+a)\psi(k,l,a)\\
&&\quad\quad\quad\quad\quad-F(a,i+j)F(i+j+a,k+l)\psi(i,j,a)\psi(k,l,i+j+a))x_{i+j+k+l+a}{\del\over\del x_a}\end{eqnarray*}
where one term comes from setting $a=k+l+b$ in the sum, and we then change variable from $b$ to $a$. The other term comes from setting $b=i+j+a$. Comparing with the desired
result for the commutator,   comparing, we require
\begin{eqnarray*}&&\kern-20pt  F(k,l)F(k+l+a,i+j)F(a,k+l)\psi(i,j,k+l+a)\psi(k,l,a)\\
&&\quad -F(k,l)F(i+j+a,k+l)F(a,i+j)\psi(k,l,i+j+a)\psi(i,j,a)\\
&=& F(k,i+j) F(i+j+k,l)\psi(i,j,k) F(a,i+j+k+l)\psi(i+j+k,l,a)\\
&&\quad -F(l,i+j) F(i+j+l,k)\psi(i,j,l) F(a,i+j+k+l)\psi(i+j+l,k,a).
\end{eqnarray*}
We divide through by $F(a,i+j+k+l)F(i+j,k+l)F(k,l)$ to obtain equivalently
\begin{eqnarray*}&&\phi(a,i+j,k+l)\left(\Rcal(i+j,k+l)\psi(i,j,k+l+a)\psi(k,l,a)-\psi(k,l,i+j+a)\psi(i,j,a)\right)\\
&=&\phi(i+j,k,l)\left(\Rcal(k,i+j)\psi(i.j.k)\psi(i+j+k,l,a)-\Rcal(k,l)\Rcal(l,i+j)\psi(i,j,l)\psi(i+j+l,k,a)\right)\end{eqnarray*}
which we are now able to prove numerically (using Mathematica) for all $i,j,k,l,a$ using the values in Lemma~2.1 and Proposition~6.10. As the commutators close, we obtain a Lie algebra realised by these vector fields.
\endproof

For $n=3$ we deduce that these are the structure constants of $g_2$ a basis formed out of 7 of the $D_{ij}$. Meanwhile, for $n=2$ and we obtain
\[ [D_{ij},D_{kl}]=4(\delta_{k,i}D_{jl}-\delta_{k,j}D_{il}-\delta_{l,i}D_{jk}+\delta_{l,j}D_{ik})\]
since if $i\ne j$ then either $k=i+j$ which gives $\psi(i,j,k)=0$, or $k=i,j$ which gives $\psi(i,j,k)=4$. Similarly for $\psi(i,j,l)$, while if $i=j$ both sides are zero. This is the Lie algebra $so(3)$ as expected in this case.

\section{Concluding remarks}

Since a Hopf algebra is a trivial (in the sense of coassociative) example of a Hopf coquasigroup, quantum groups such as $\C_q[SU_2]$ could be viewed as such. Just as we have seen above that the Lie algebra $su_2$ is naturally obtained in terms of the $F$ structure constants on $\Z_2^2$, we have seen how $k[S^3]$ can likewise be obtained as (in this case) a Hopf algebra. This Hopf algebra with coproduct from Proposition~5.7 looks, however, very different from the usual matrix coproduct used to describe the algebraic group $\C[SU_2]$ and its $q$-deformation. We also have to describe the $*$-involution structure in the case over $\C$ that picks out the compact real form $SU_2=S^3$. 

We do this as follows. Referring to $k[S^{2^n-1}]$ we describe elements of $\Z_2^n=\Z_2^{n-1}\times\Z_2$ labelling the generators in the form $a0$ or $a1$ where $a\in \Z_2^{n-1}$. We also recall that the $F$ for the octonions is built up by the Cayley-Dickson process from $F$ on $\Z_2^{n-1}$. This is explained in \cite{Ma99}
and amounts to a formula which we now write as
\[ F(a0,b0)=F(a,b),\quad F(a0,b1)=F(a,a)F(a,b)\] \[  F(a1,b0)=\Rcal(a,b)F(a,b),\quad
 F(a1,b1)=-F(a,a)\Rcal(a,b)F(a,b).\]
We now define complex generators for $\C[S^{2^n-1}]$ by 
\[ z_a=x_{a0}+\imath x_{a1}\]
Then after a short computation the coproduct of $\C[S^{2^n-1}]$ in Proposition~5.7 takes the form 
\[ \Delta z_a=\sum_{b+c=a}F(b,c) z_b^{\Rcal(b,c)}\tens z_c^{F(b,b)},\quad z_a^{\pm 1}=x_{a0}\pm\imath x_{a1}\]
(so $z_a^\eps$ denotes  $z_a$ if $\eps=1$ and $z_a^*$ if $\eps=-1$). Thus $\C[S^{2^n-1}]$ is the commutative polynomial algebra in complex generators $z_a,z_a^*$ with relations 
\[ \sum z_a z_a^*=1,\]
the above coproduct and a $*$-involution sending $z_a$ to $z_a^*$. For $S^3$ it means two complex generators $z_0$ and $z_1$ with $F(1,1)=-1$ the cochain for $\C$, i.e. the coproduct
\[ \Delta z_0=z_0\tens z_0-z_1\tens z_1^*,\quad \Delta z_1=z_0\tens z_1+z_1\tens z_0^*\]
and antipode $Sz_0=z_0^*$, $Sz_1=-z_1$. We can think of this as an $SU_2$ matrix of generators $\begin{pmatrix}z_0^*& -z_1^*\\ z_1 & z_0\end{pmatrix}.$

We can now allow noncommutation relations between the generators. The relations
\[ z_0z_1=qz_1z_0,\quad z_0^*z_1=z_1z_0^*\quad z_0^*z_0=z_0z_0^*+(q-q^{-1})z_1z_1^*\]
and the sphere relation $z_0z_0^*+z_1z_1^*=1$ defines the quantum group $\C_q[S^3]$ as a $*$-Hopf algebra for real $q$. The relation with the usual matrix coproduct is now $z_0=d$ and $z_1=q^{-{1\over 2}}c$ in terms of the usual generators. 

We can similarly compute the coproduct of $\C[S^7]$  using the explicit form of $F$ for the quaternions to obtain
\[ \Delta z_i=z_i\tens z_0^*+z_0\tens z_i+\sum_{j,k}\eps_{ijk}z_j^*\tens z_j^*,\quad  \Delta z_0=z_0\tens z_0-\sum_i  z_i\tens z_i^*\]
where now $i,j,k\in\{1,2,3\}$ and $\eps$ is the totally antisymmetric tensor with $\eps_{123}=1$. This gives our explicit form of this Hopf coquasigroup in terms of complex $*$-algebra generators.  There are also several ideas for noncommutative and $q$-deformed $S^7$ algebras, notably \cite{VS}, and some of these may be compatible with the above coproduct. This is a topic for a sequel. 

\subsection*{Acknowledgements} Most of these results were presented at the 2nd Mile High conference, Denver, June 2009. We are, however, grateful to M. Bremner  for pointing out the reference \cite{PS} (and for an interesting conference presentation) leading us to add Propositions~4.8 and~4.9.

\end{document}